\documentclass[a4paper,10pt,leqno,english]{article}
\usepackage{tikz,stix}
\usetikzlibrary{positioning}
\usepackage[T1]{fontenc}
\usepackage[utf8]{inputenc}
\usepackage[a4paper, margin=2cm]{geometry}
\usepackage{babel}
\usepackage{amsthm,amsmath,amsfonts,amssymb}
\usepackage{dsfont}
\usepackage{enumitem,pifont,physics,comment}
\usepackage{hyperref}
\usepackage{float}
\newtheorem{theorem}{Theorem}
\newtheorem{proposition}{Proposition}
\newtheorem{lemma}{Lemma}
\newtheorem{corollary}{Corollary}
\newtheorem{examples}{Examples}

\newtheorem{definition}{Definition}[section]

\newtheorem{remark}{Remark}

\usepackage[backend=biber,style=alphabetic]{biblatex}
\usepackage{csquotes}
\usepackage{amsmath,amsthm,amssymb,amscd,color, xcolor,mathtools,url,tikz}
\usepackage{tikz}
\usetikzlibrary{positioning}

\usepackage{titlesec}

\def\R{\mathbb{R}}
\def\N{\mathbb{N}}
\def\Z{\mathbb{Z}}

\def\T{\mathbb{T}}

\let\eps\varepsilon
\let\epsilon\varepsilon

\let\oldsum\sum
\renewcommand{\sum}{\displaystyle\oldsum}
\let\oldprod\prod
\renewcommand{\prod}{\displaystyle\oldprod}
\let\oldinf\inf
\renewcommand{\inf}{\displaystyle\oldinf}
\let\oldsup\sup
\renewcommand{\sup}{\displaystyle\oldsup}

\let\leq\leqslant
\let\geq\geqslant
\newlength{\oldparindent}
\newcommand{\myindent}{\hspace{\oldparindent}}

\usepackage{etoolbox}

\title{$L^4$ norm of spectral projectors on polynomially small frequency intervals for $S^1$-symmetric surfaces}
\author{Ambre Chabert\thanks{Département de Mathématiques et Applications, Ecole Normale Supérieure, UMR 8553, 45 rue d'Ulm. Bureau C11. 75230 PARIS Cedex 05, France. Université Paris Cité and Sorbonne Université, CNRS, IMJ-PRG, F-75013 Paris, France. Email adress : ambre.chabert@ens.fr}}

\begin{document}
\maketitle

\begin{abstract}
    For $(M,g)$ a compact Riemannian surface with Laplace-Beltrami operator $\Delta$, and for $\lambda,\delta \geq 0$, let $P_{\lambda,\delta}$ be the spectral projector on the frequency interval $[\lambda-\delta,\lambda+\delta]$ associated to $\sqrt{-\Delta}$. For the Euclidean disk, away from its boundary, we improve the upper bound on the $L^2\to L^4$ norm of $P_{\lambda,\delta}$ in the regime where the bandwidth $\delta$ is polynomially small compared to the target frequency $\lambda$. Decomposing on the explicit joint eigenbasis of $\left(\sqrt{-\Delta}, \frac{1}{i}\frac{\partial}{\partial \theta}\right)$ given in terms of Bessel eigenfunctions, which are well-approximated by oscillatory functions outside of their caustic set, we reduce the analysis to a number of precise quantitative estimates of nonstationary phase oscillatory integrals. We strongly use convexity phenomenon both for these estimates, and then for the summation of the contribution of all eigenfunctions through a new arithmetic estimate. The method extends to other $S^1$-symmetric surfaces satisfying similar conditions on the induced completely integrable structure.
\end{abstract}

\section{Introduction}\label{secintro}

\myindent Let $(M,g)$ be a compact Riemannian manifold, with Laplace-Beltrami operator $\Delta$ (in the case $\partial M \neq 0$, we fix Dirichlet boundary conditions). For $\lambda,\delta \geq 0$, let $P_{\lambda,\delta}$ be the
\textit{spectral projector} associated to $\sqrt{-\Delta}$ on the frequency interval $[\lambda-\delta,\lambda +\delta]$, which may be defined through functional calculus as
\[P_{\lambda,\delta} := 1_{[\lambda-\delta, \lambda + \delta]} (\sqrt{-\Delta}).\]

\myindent The question of estimating the operator norm of $P_{\lambda,\delta}$, seen as an operator from $L^2$ to $L^p$, $2\leq p \leq +\infty$, intertwines precise estimates of the spatial distribution of the eigenfunctions and of the spectrum of the Laplacian, and the geometry of $M$. While the case $\delta = 1$ is fully understood since the pioneering works of Sogge \cite{sogge1988concerning}, and is purely local, the case $\delta \ll 1$, which depends on the global geometry, is still open for most settings, see the review of Germain \cite{germain2023l2}. In this article, we investigate the case $p = 4$, which corresponds to the \textit{geodesic focusing regime} (see below), for the Euclidean disk, in the regime where the bandwidth $\delta$ is \textit{polynomially small} compared to the target frequency $\lambda$. Moreover, we argue how to extend these results to other $S^1$ symmetric surfaces, after microlocalizing on \textit{regular} open sets of the cotangent space, as defined below. We will also explain how the case of the disk illustrates the role of \textit{concentration on the caustics} in the theory of $L^p$ norm of spectral projectors.

\subsection{Main result}\label{subsecmaint}

\myindent Let $\mathbb{D}^2 = D(0,1) \subset \R^2$ be the unit disk with flat metric. We denote $p(x,\xi)$ the norm of cotangent vectors $(x,\xi) \in T^*\mathbb{D}^2$, which is the principal symbol of $\sqrt{-\Delta}$. If, on $\mathbb{D}^2 \backslash 0$, we set the polar coordinates $(r,\theta)$, we define for $(x,\xi) \in T^*\mathbb{D}^2$ the \textit{angular momentum} $\Theta(x,\xi)$ as the dual variable of $\theta$. We prove the following theorem.

\begin{theorem}\label{main}
    Let $\eps, \eta > 0$. Let $\chi(x,D)$ be a pseudodifferential microlocalizer inside the cone 
    \[\{(x,\xi) \in T^*\mathbb{D}^2 \qquad \Theta(x,\xi) \leq (1 - \eps/2) p(x,\xi)\}.\]
    \myindent Then, there holds, for any $\lambda \geq 1$,
    \[\forall \delta \geq \lambda^{-1/3 + \eta}, \qquad \left\|P_{\lambda,\delta} \chi(x,D)\right\|_{L^2(\mathbb{D}^2) \to L^4(\mathbb{D}^2)} \lesssim_{\eps,\eta} \lambda^{1/8} \delta^{1/8} \ln(\lambda)^{1/4}.\]
\end{theorem}

\myindent Before giving some corollaries of the theorem, let us make some comments.
\begin{itemize}
    \item The exponent $1/8$ corresponds to Sogge's exponent $\gamma(4)$, as define below. Hence, our estimate is nearly coherent with the case $\delta = 1$, with a small logarithmic loss. In particular, the exponent on $\lambda$ is optimal. The optimal exponent on $\delta$ is bounded from above by $1/2$ (see below), but we have no reason to believe that $1/8$ is optimal. However, from a theoretical point of view, it is less important to find the optimal exponent on $\delta$, since any positive exponent already yields polynomial improvements over the case $\delta = 1$.
    \item It is necessary to exclude a small cone of the cotangent space due to the presence of \textit{whispering modes}, whose $L^4$ norm is of order $\lambda^{1/6}$ (see below), and which are microlocalized near the cotangent of $\partial \mathbb{D}^2$, which is parameterized by the equation $\Theta(x,\xi) = p(x,\xi)$.
    \item However, since the whispering modes are microlocalized on a neighborhood of size $O(\lambda^{-2/3})$ of $T^*\mathbb{D}^2$, it is likely that our estimate can be relaxed by choosing $\eps = \lambda^{-\iota}$ for some $\iota > O$. In fact, in our estimates, the implicit dependency in $\eps$ can always be bounded by $\eps^{-M}$ for some large number $M$. Thus, our theorem would still hold for $\eps = \lambda^{-\iota}$, provided $0 < \iota \ll 1$ is sufficiently small, and loosing a very small power of $\lambda$. However, we will not quantify this phenomenon since it wouldn't give the optimal $\iota$.
    \item With regards to the lower bound on the bandwidth $\lambda^{-1/3 + \eta}$, it comes from two restrictions. First, a spectral restriction on the remainder of the Weyl law for the Euclidean disk (see below \eqref{improvedremainder}), i.e. from the smallest scale on which the eigenvalues of $\sqrt{-\Delta}$ are well distributed. This restriction seems very hard to overcome since it is ultimately linked to the fundamental arithmetic question of the number of integers points inside a family of homothetic domains The second restriction is given by the spacing of points in the joint spectrum of $\sqrt{-\Delta}$ and of the infinitesimal generator of rotation $-i\frac{\partial}{\partial \theta}$. Now, this restriction can be avoided in the regime $\delta \leq \lambda^{-2/3}$, and, in particular, in the spectrally exact case i.e. for eigenfunctions, using for example Jarnik's Lemma \cite{jarnik1926gitterpunkte}. However, we will note pursue this in the present article.
\end{itemize}

\myindent We now give some straightforward corollaries of Theorem \ref{main}. Choosing $\chi(x,D)$ such that it equals $1$ on the smaller cone
\[\{(x,\xi) \in T^*\mathbb{D}^2 \qquad \Theta(x,\xi) \leq (1 - \eps) p(x,\xi)\},\]
which contains in particular $T^*D(0,1-\eps)$, the theorem implies the following estimate.

\begin{corollary}
    Let $\eps,\eta > 0$. For any $\lambda \geq 1$, there holds
    \[\forall \delta \geq \lambda^{-1/3 + \eta} \qquad \left\|P_{\lambda,\delta}\right\|_{L^2(\mathbb{D}^2) \to L^4(D(0,1-\eps))} \lesssim_{\eps,\eta} \lambda^{\frac{1}{8}}\delta^{\frac{1}{8}} \ln(\lambda)^{1/4}\]
\end{corollary}

\myindent Moreover, since any eigenfunction with eigenvalue $-\lambda^2$ (i.e. $\Delta \phi_\lambda = -\lambda^2 \phi_\lambda$) is in the range of $P_{\lambda,\delta}$, we find the following corollary.
\begin{corollary}
    Let $\eps,\eta > 0$, and let $\phi_{\lambda}$ be a $L^2$ normalized eigenfunction of $\Delta$ with eigenvalue $-\lambda^2$. Then, with $\chi(x,D)$ as in Theorem \ref{main}, there holds
    \[\|\chi(x,D) \phi_\lambda\|_{L^4(\mathbb{D}^2)} \lesssim_{\eps,\eta} \lambda^{\frac{1}{12} + \eta},\]
    and thus
    \[\|\phi_\lambda\|_{L^4(D(0,1-\eps))} \lesssim_{\eps,\eta} \lambda^{\frac{1}{12} + \eta}.\]
\end{corollary}

\subsection{Extension to surfaces isometric under a $S^1$ action}

\myindent Our method easily extends to prove similar bounds in the context of compact Riemannian surfaces which are isometric under a $S^1$ group action, with the cost of microlocalizing in a conic subset of the cotangent space which is \textit{regular} enough, in the sense precised below. While we won't give a full proof of this result, since the method is similar, we will nonetheless argue in Appendix \ref{Appext} how one can reduce to the same computations than for the disk.

\myindent Let $\Sigma$ be a compact Riemannian surface, which is isometric under a $S^1$ action, whose fixed points are isolated. Let $P$ be the infinitesimal generator of the action, which is a pseudodifferential operator of order $1$, and let $\Theta : T^*\Sigma\backslash 0 \to \R$ be its principal symbol. Observe that $\Theta$ is preserved by the geodesic flow in $T^*\Sigma \backslash 0$, which is the Hamiltonian flow of the norm of cotangent vectors $p:T^*\Sigma \backslash 0 \to \R^*_+$. In particular, the Poisson bracket $\{p,\Theta\}$ vanishes, and the geodesic flow is thus \textit{completely integrable} on the open conic set $\Omega \subset T^*\Sigma \backslash 0$ where $dp,d\Theta$ are linearly independent. We assume that the complementary set of $\Omega$ has codimension at least 1, so that $\Sigma$ is \textit{completely integrable} in the sense of \cite{de1977quasi}[Section 4]. Let $M = (p,\Theta) : \Omega \backslash 0 \to \R^2 \backslash 0$ be the momentum application. For $a\in M(\Omega)$, the level sets $\Lambda_a := M^{-1}(a)$ of $M$ are compact lagrangian submanifolds of $T^*\Sigma\backslash 0$, which are diffeomorphic to $2D$ tori. We recall the fundamental theorem of \textit{action-angle coordinates} in this setting, see \cite{de1977quasi}[Theorem 4.1]

\begin{theorem}[Colin de Verdière]\label{thmCdV}
    For all $a \in M(\Omega)$, there exists a conic convex neighborhood $U\subset \R^2\backslash 0$ of $a$, and an open connected cone $C \subset \R^2\backslash 0$, and a canonical transformation $\psi$ from $\mathbb{T}^2 \times C \subset T^*\mathbb{T}^2 \backslash 0$ to $M^{-1}(U)$ such that 
    
    i. $\psi^{-1}$ sends the foliation $(\Lambda_a)_{a\in U}$ to the foliation $(T_\xi)_{\xi\in C}$, with $T_\xi = \mathbb{T}^2 \times \{\xi\}$.
    
    ii. There holds $p \circ \psi(x,\xi) = K(\xi)$, where $K$ is a smooth homogeneous function of order 1 on $C$.
\end{theorem}

\myindent Observe that we may fix the angular momentum $\Theta = \xi_2$ as one of the action coordinates, since its Hamiltonian flow is already $2\pi$-periodic. We will always make this choice in the following.

\myindent Now, in general, the Lagrangian torus $\Lambda_a$ doesn't project smoothly on $\Sigma$ due to the presence of \textit{caustics}. We consider the simplest case of caustic, that is \textit{fold-type caustics}.

\begin{definition}
    Let $\Lambda \subset T^*\Sigma \backslash 0$ be a compact Lagrangian submanifold, let $\Pi : T^*M\Sigma\to \Sigma$ be the canonical projection, and let $\pi : \Lambda \to \Sigma$ be its restriction to $\Lambda$. Following \cite{golubitsky1980stable}, we say that $p\in \Lambda$ is singular if $d\pi(p)$ has a nonzero kernel, and we define $S\subset \Lambda$ the set of singular point. We recall that the {\rm caustic set} of $\Lambda$ is defined as $\pi(S)$. 
    
   \myindent Assume that $S$ is a submanifold of $\Lambda$ of dimension $1$. We say that $p\in S$ is of {\rm fold-type} if $ker(d\pi(p))$ and $T_p S$ span together $T_p \Lambda$. Finally, we say that $\Lambda$ has {\rm fold-type singularities} if all $p\in S$ are of fold-type. 
\end{definition}

\begin{definition}\label{defregdom}
    Let $\omega \subset \Omega$ be a convex open conic domain of action-angle coordinates. In particular, let us write $\omega = \psi(\mathbb{T}^2 \times C)$, with $C$ an open convex cone in $\R^2\backslash 0$. Thus, $\omega$ is foliated by the Lagrangian tori $L_\xi := \psi(\mathbb{T}^2\times \{\xi\})$, with $\xi \in C$. For $\xi \in C$, write $S_\xi$ the set of singular points of $L_\xi$. We say that $\omega$ is a \textit{regular domain} if the following conditions are satisfied
    
    i. The curve $\{K = 1\}$ has everywhere nonzero curvature on $C$.
    
    ii. There is a finite number of directions, say $\R_+^* \xi_1,...,\R_+^* \xi_K \subset C$, such that, for all $\xi \in C \backslash \{\R_+^* \xi_1,...,\R_+^* \xi_K\}$, $\Lambda_\xi$ has fold-type singularities. 
\end{definition}

\begin{definition}\label{defadapted}
    With the notations of Definition \ref{defregdom}, if $\chi(x,D)$ is a pseudodifferential microlocalizer supported inside $\omega$, we say that $\chi$ is \textit{adapted} to $\omega$ if
    
    i. There is an open subcone $C_1 \subset \bar{C_1} \subset C$ such that $\chi = 0$ outside of $\psi(\mathbb{T}^2 \times C_1)$.

    ii. There is a constant $c > 0$ such that, for all $i = 1,...,K$, $d(S_{\xi_i}, supp(\chi(x,D)) \geq c$.
\end{definition}

\myindent The interest of the definition is that Theorem \ref{main} extends after microlocalizing on a regular $\omega$.

\begin{theorem}\label{thmext}
    Let $\omega$ be a regular domain as in Definition \ref{defregdom}. Let $\chi(x,D)$ be a pseudodifferential microlocalizer adapted to $\omega$.  Then, there holds, for any $\eta > 0$,
    \[\forall \delta \geq \lambda^{-1/3 + \eta}, \qquad \left\|P_{\lambda,\delta} \chi(x,D) \right\|_{L^2(M) \to L^4(M)} \lesssim \lambda^{1/8} \delta^{1/8} \ln(\lambda)^{1/4}.\]
\end{theorem}

\begin{remark}
    This theorem doesn't cover Theorem \ref{main} since it would typically only apply away from the center and the boundary of the Euclidean disk $\mathbb{D}^2$. Hence, we give a direct proof in the case of the disk in order to include the center and the boundary in our estimates.
\end{remark}

\myindent Let us conclude this section by giving two main examples of the above situation

\begin{examples}
    \begin{itemize}
        \item In the case where $\Sigma = \mathbb{S}^2$ is the two-sphere equipped with a metric of revolution, which is \textit{simple} (see \cite{de1980spectre}) and satisfies the \textit{twist hypothesis} (see \cite{bleher1994distribution}), then $\Omega$ equals the cotangent space of $\Sigma$ from which is removed the two-dimensional subsets of cotangent vectors of the equator. Moreover, $\Omega$ is itself a regular domain, since the twist hypothesis guarantees the nonzero curvature assumption, and all Lagrangian tori have fold-type singularities, except for the Lagrangian torus of vertical cotangent vectors. Thus, for any compact subset $K$ of $\Sigma$ which doesn't contain a pole or intersects the equator, Theorem \ref{thmext} gives a bound on the $L^2(\Sigma) \to L^4(K)$ norm of spectral projectors. For a justification of the above facts, one may consult \cite{chabert2025infty}, or \cite{chabert2025bounds} for a more detailed introduction.
        
        \item Let $\Sigma = \mathbb{T}^2$ be a 2D-torus equipped with a metric of revolution, which we may write as $g = (dx^1)^2 + f(x^1) (dx^2)^2$ in canonical coordinates $(x^1,x^2)$. Assume that either $f$ is constant, either that its stationary points are isolated. Then, the above theorem still applies after microlocalizing away from the vectors which are cotangent to the 1-dimensional submanifold $\{f' = 0\}$, and away from generically exceptional Lagrangian submanifolds for which the curvature of the curve $\{K =1\}$ vanishes.
    \end{itemize}
\end{examples}

\subsection{Background}\label{subsecbiblio}

\myindent The first upper bound for the $L^{\infty}$ norm of eigenfunctions on a compact Riemannian manifold was given by Hörmander in \cite{hormander1968spectral}. It was generalized by Sogge to spectral projectors and $L^p$ norms in \cite{sogge1988concerning} to the following result: let $(M,g)$ be a complete boundaryless Riemannian manifold of dimension $d$. Then, for all $\lambda \geq 1$, and all $\infty \geq p \geq 2$, there holds
\[\|P_{\lambda,1} \|_{L^2(M) \to L^p(M)} \lesssim_M \lambda^{\gamma(p)},\]
where the exponent $\gamma(p)$ is given by
\[\gamma(p) = \begin{cases} 
\frac{d-1}{2} - \frac{d}{p} \qquad p_c \leq p \leq \infty\\
\frac{d-1}{2}\left(\frac{1}{2} - \frac{1}{p}\right) \qquad 2 \leq p \leq p_C\end{cases},\]
the two behaviors being separated by the \textit{Stein-Thomas} exponent $p_C = 2\frac{d+1}{d-1}$. Moreover, this upper bound is optimal on all manifolds, in the sense that the lower bound also holds, up to changing $\lambda$ to a $\lambda_0$ satisfying $|\lambda - \lambda_0| \leq R_0$ for some fixed $R_0(M) > 0$. As a consequence, one finds the upper bound $\|\phi_\lambda\|_{L^p(M)}\lesssim \lambda^{\gamma(p)}$ for $L^2$-normalized eigenfunctions $-\Delta \phi_\lambda = \lambda^2 \phi_\lambda$, which corresponds to the case $\delta= 0$. In general, this upper bound is sharp, since it is saturated on the sphere $S^d$ by \textit{zonal harmonics} (resp \textit{highest weight harmonics}) for the \textit{point-focusing regime} $p_C \leq p \leq \infty$ (resp the \textit{geodesic-focusing regime} $2\leq p \leq p_C$). However, in general, one expects that the $L^p$ norm of eigenfunctions is much smaller (see \cite{sogge2002riemannian}). The study of the $L^2\to L^p$ norm of the spectral projector $P_{\lambda,\delta}$ on a small frequency interval can be seen as an intermediate problem between the mostly open question of $L^p$ norm of eigenfunctions, and the fully understood case $\delta =1$. 

\myindent One of the first improved $L^p$ bound result is the well-known result of Cooke-Zygmund (see \cite{cooke1971cantor,zygmund1974fourier}) that, on the 2D-torus $\T^2 := \R^2 / (2\pi \Z)^2$, any $L^2$-normalized eigenfunction of the Laplacian satisfies $\|\phi\|_{L^4(\T^2)}\lesssim 1$. It was generalized by Bourgain \cite{bourgain1993eigenfunction} to the conjecture that, on the torus $\T^d := \R^d / (2\pi \Z)^d$, there holds
\[ \|\phi_{\lambda}\|_{L^p(\mathbb{T}^d)} \lesssim \lambda^{\frac{d}{2} - 1 - \frac{d}{p}} \qquad p > \frac{2d}{d-2}.\]

\myindent Particular cases of this conjecture were proved in \cite{bourgain2013moment}, culminating with the $l^2$ decoupling theorem, see \cite{bourgain2015proof}. More generally, estimates on spectral projectors on thin frequency intervals for tori have been extensively studied in \cite{germain2022boundsBook,germain2022bounds,demeter2024l2,hickman2020uniform}.

\myindent For \textit{arithmetic surfaces}, it is conjectured in \cite{iwaniec1995norms} that there holds $\|\phi_{\lambda}\|_{L^{\infty}}\lesssim \lambda^{\eps}$ for any $\eps > 0$ (hence the same bound for $L^p$ norms). For recent progresses on that conjecture, see \cite{buttcane2017fourth,humphries2018equidistribution,humphries2022p}.

\myindent More generally, for manifolds with \textit{nonpositive curvature}, many works have obtained \textit{logarithmic} improvements on the norm of eigenfunctions, see \cite{berard1977wave,hassell2015improvement,hezari2016lp,blair2017refined,blair2018concerning,blair2019logarithmic}. 

\myindent The converse question of finding geometric assumption on $M$ under which there exists sequences of eigenfunctions with high $L^p$ norms has been studied in many works, such as \cite{sogge2002riemannian,sogge2011blowup,sogge2016focal,sogge2016focalb,sogge2001riemannian,canzani2019growth,canzani2021eigenfunction}.

\myindent Finally, we mention that, for the $L^p$ norm of eigenfunctions of the \textit{magnetic Laplacian} on compact hyperbolic surface, one observes three of the main behaviors, depending on energy levels \cite{chabert2026zonalstatesimprovedlinfty}: for low energy levels, the situation is similar to the sphere, that is Hörmander's bound is saturated. For the critical energy level, a polynomial improvement holds. Finally, for high energy levels, the situation is similar to manifolds with nonpositive curvature.

\quad

\myindent In the case of compact manifolds with a boundary, with Dirichlet or Neumann boundary condition, it was observed by Grieser \cite{grieser1992p} that, on the Euclidean disk, the $L^2$-normalized \textit{whispering modes}, which are eigenfunctions that concentrates on a $\lambda^{-2/3}$-neighborhood of the boundary, have $L^6$ norm $\lambda^{2/9}$, where $2/9 > \gamma(6) = 1/6$. More generally, it was proved in \cite{smith2007p} that, when $M$ is a compact surface with boundary, there holds 
\[\|P_{\lambda,1} \|_{L^2(M)\to L^p(M)} \lesssim_M \lambda^{\mu(p)},\]
where 
\[\mu(p) = \begin{cases}
2\left(\frac{1}{2} - \frac{1}{q}\right) - \frac{1}{2} \qquad 8 \leq p \leq \infty\\
\frac{2}{3}\left(\frac{1}{2} - \frac{1}{q} \right) \qquad 2 \leq p \leq 8\end{cases},\]
and the upper bound is saturated for the Euclidean disk. In particular, this explains why it is necessary to exclude a small neighborhood of $\partial\mathbb{D}^2$ in Theorem \ref{main}. However, since Sogge's bound is purely local, the bound for boundaryless compact manifolds still holds on any compact set $K\subset Int(\mathbb{D}^2)$. This is why the exponent $1/8$ in Theorem \ref{main} corresponds to $\gamma(4)$ and not to $\mu(4) = 1/6$. We intend to study in a future work the transitional regime between $\gamma(p)$ and $\mu(p)$ for the norm of $P_{\lambda,1}$ on the disk, when microlocalizing on dyadic neighborhoods of the boundary, in the manner of \cite{Koch2005}.

\myindent Let us comment a little on the possible results for any values of $p,\delta$. First, Colin de Verdière proved in \cite{colin2010remainder} a polynomial improvement for the remainder of the Weyl law for the Euclidean disk: if $0\leq\lambda_0^2 \leq \lambda_1^2 \leq ...$ are the ordered eigenvalues of $-\Delta$, then, for $\lambda \geq 0$,
\[N(\lambda) := \#\{j\geq 0 \ \text{such that} \ \lambda_j \leq \lambda\} = c\lambda^2 + O(\lambda^{2/3}),\]
which is a strong indication that one can obtain improved upper bounds for spectral projectors on polynomially thin frequency intervals. Indeed, this ensures that
\begin{equation}\label{improvedremainder}
    \forall \delta \geq \lambda^{-1/3}, \qquad \#\{j\geq 0 \ \text{such that} \ \lambda_j \in [\lambda - \delta, \lambda + \delta]\} \lesssim \lambda \delta.
\end{equation}

\myindent Now, let us observe that, for any point $x_0 = (r_0,\theta_0)$ with $0 < r_0 < 1$, the set \[\{(x,\xi) \in T^*\mathbb{D}^2\qquad \Theta(x,\xi) \geq r_0\}\] is locally a Lagrangian submanifold around $x_0$, with a fold-type caustic on the circle of radius $r_0$. Hence, from the standard construction of $O(h^{\infty})$ quasimodes on a Lagrangian submanifold (see \cite{duistermaat1974oscillatory}), and their expression in terms of the Airy function, one can construct a $O(\lambda^{-\infty})$ quasimode around $x_0$, which is of size $\lambda^{1/6}$ on an annulus around the circle of radius $r_0$, and of width $\lambda^{-2/3}$. In fact, from the explicit expression of eigenfunctions in terms of Bessel functions (see below), if $r_0$ belongs to the dense subset of $[0,1]$ formed by the $\mu_{k,n}$ in Definition \ref{defJn}, this is an exact eigenfunction. Now, from this construction, we find that, for any $N\geq 0$, and for any $K\subset \mathbb{D}^2$ of nonempty interior,
\[\forall p > 4 \qquad \|P_{\lambda,\lambda^{-N}}\|_{L^2(\mathbb{D})^2 \to L^p(K)} \gtrsim \lambda^{\frac{1}{6} - \frac{2}{3p}},\]
which is better than the a priori universal lower bound $\lambda^{\gamma(p)} \delta^{1/2}$ for $P_{\lambda,\delta}$ in the regime $\delta \leq \lambda^{-\left(\frac{2}{3} - \frac{8}{3p}\right)}$.

\myindent The case $p = 4$ is singular in the sense that, as can be seen from direct computation, the above quasimodes (and the exact eigenfunctions given by Bessel functions) have $L^4$ norm of order $\ln(\lambda)$. In particular, the Cooke-Zygmund estimate doesn't hold on the disk, and we see a major difference in the geometry of the disk, compared to the flat torus, given by the presence of \textit{caustics}. Hence, we start by studying the case $p = 4$ in the present article, and we will study the critical exponent $p = 6$ in a future work.

\myindent In the \textit{point-focusing regime} $p > 6$, the author obtained with Colin de Verdière improved upper bound on polynomially small frequency intervals in the closely related \cite{chabert2025infty} for $p = \infty$ (hence for any $p > 6$ by interpolation). In particular, the method of proof of the present article is directly inspired by this paper.

\quad

\myindent Finally, the Euclidean disk is an example of a manifold invariant under an $S^1$ isometric action, as was studied in \cite{donnelly2001bounds}. In particular, it falls in the more general class of \textit{completely integrable} manifolds, for which Bourgain predicted that polynomial improvement on the $L^{\infty}$ norm of eigenfunctions holds in \cite{bourgain1993eigenfunctiona}. We mention that we were able to quantify this phenomenon in terms of spectral projectors for some surfaces of revolution in \cite{chabert2025bounds}.

\subsection{Outline of the proof}\label{subsecidea}

\myindent Following the approach of \cite{chabert2025infty}, we start in Section \ref{secnotations} by decomposing on the joint eigenbasis of $\sqrt{-\Delta}$ and the generator of infinitesimal rotation $\frac{1}{i}\frac{\partial}{\partial \theta}$, which gives eigenfunctions of the form
\[\phi_{k,n}(r,\theta) = c_{k,n} \lambda_{k,n}^{1/2} J_n(\lambda_{k,n} r) e^{in\theta},\]
where $J_n$ is the $n$-th Bessel function, for integers $k,n$, constants $c_{k,n} \simeq 1$, and where $\lambda_{k,n}$ is the eigenvalue for $\sqrt{-\Delta}$. Thus, all amounts to estimating the coefficients $C_{\textbf{n}}$ defined as the integral of the products of four such eigenfunctions indexed by quadruplets $\textbf{n} = (n_1,n_2,n_3,n_4)$ (the condition that $\lambda_{k,n} \in [\lambda-\delta,\lambda +\delta]$ fixes $k = k(n)$). Here, we crucially use the explicit expression of the $L^4$ norm (which holds for any $L^p$ norm with $p$ an even integer) in terms of products.

\myindent In Section \ref{secreduc}, we first explain how the integral of the product vanishes outside of a simple \textit{zero sum set condition} on the quadruplet $\textbf{n}$, which must satisfy $n_1 + n_4 = n_2 + n_3$. In particular, this may be seen as a reduction of the "dimension" of the set of quadruplets $\textbf{n}$, i.e. one has at most three free parameters $n$. Then, we use the explicit form of Bessel eigenfunctions in terms of Lagrangian oscillatory integrals, for which we give precise quantitative estimates in Appendix \ref{AppBessel}, in order to reduce the problem to a number of oscillatory integrals estimate, with a large parameter $\lambda$. The interest is that we not only exploit the decay of joint eigenfunctions outside of caustics, but also their oscillatory nature with increasing frequency the furthest away from the caustic. The idea is then that, for "most" value of the quadruplet $\textbf{n}$, and outside of very small neighborhoods of their caustics, the eigenfunctions are well approximated in terms of simple oscillatory functions, but with different frequencies for each. The point is that, when integrating the product of oscillating functions with different frequencies, one gains averaging effects through integration by parts.

\myindent In Section \ref{secmain}, we bound the main contribution in the integral of the product of four eigenfunctions $C_\textbf{n}$. Indeed, since our analysis relies mostly on careful and quantitative integration by parts, all is governed by lower bounds on the gradient of the phases which appear in the oscillatory integrals. Now, the most delicate case for proving lower bounds on this gradient is given by a certain \textit{convexity} condition, where even the fact that the gradient doesn't vanish relies crucially on a convexity property of the oscillatory frequency of Bessel functions away from their caustic, which is the equivalent of the fact that the Airy functions oscillates as a $\cos(|x|^{3/2})$ when $x\to -\infty$.

\myindent Then, one needs to sum the bound obtained for a fixed quadruplet $\textbf{n}$ on all quadruplets satisfying the zero sum set condition. This relies on a new arithmetic estimate on the distribution of points which are both in a small neighborhood of a strictly convex curve, and on the rescaled lattice  $\lambda^{-1}\Z^2$ This estimate can be seen as an extension of the classical \textit{Jarnik's estimate} to \textit{quadrilaterals} with vertices satisfying a certain resonant condition.

\myindent Finally, in Section \ref{secremaind}, we prove that all of the remaining terms are indeed small compared to the main order terms.

\section{Notations and first reductions}\label{secnotations}

\myindent We start by introducing the explicit basis of joint diagonalization of the Laplacian on $\mathbb{D}^2$ and of the generator of the rotational symmetry $\frac{1}{i} \frac{\partial}{\partial \theta}$ given in terms of Bessel function. 

\begin{definition}
    For $n \in \Z$, we introduce the Bessel function
    \begin{equation}\label{defJn}
        J_n(t) := \frac{1}{2\pi} \int_{\mathbb{T}} e^{i(n\alpha - t\sin \alpha)} d\alpha.
    \end{equation}
    
    \myindent If, for $k\geq 0$, we denote $\lambda_{k,n} > 0$ the $k$th zero of $J_n$, we recall that an $L^2$ orthonormal basis of eigenfunctions of the Laplacian with Dirichlet boundary condition on the disk is given, in polar coordinates, by
    \[\phi_{k,n}(r,\theta) = c_{k,n} \lambda_{k,n}^{1/2} J_n(\lambda_{k,n} r) e^{in\theta},\]
    where $c_{k,n}$ are positive normalization constants, and where the associated eigenvalue of $\Delta$ is $-\lambda_{k,n}^2$. Moreover, there holds $c_n \simeq (1 - \mu_{k,n})^{-1/4}$, where $\mu_{k,n} := \frac{|n|}{\lambda_{k,n}}$ is the radius at which the caustic set of $\phi_{k,n}$ is localized (see \cite{chabert2025infty}).
\end{definition}

\myindent Now, since $J_{-n}(t) = (-1)^nJ_n(t)$, we can restrict ourselves in the following to $n\geq 0$. Since we localize on the support of $\chi(x,D)$, we can furthermore restrict to $\mu_{k,n} \leq 1 - \eps/2$ (see \cite{chabert2025infty}). In particular, we can reduce the problem to the following: let
\[\mathcal{Z} := \{k,n\geq0 \qquad\ \text{such that}\ \lambda_{k,n} \in [\lambda - \delta, \lambda + \delta] \ \text{and} \ n \leq (1- \eps/2) \lambda_{k,n}\}.\]

\myindent Then, fix
\[u := \sum_{(k,n) \in \mathcal{Z}} a_{k,n} \phi_{k,n} \qquad \text{such that} \ \sum |a_{k,n}|^2 = 1.\]

\myindent We need only prove that
\[\|u\|_{L^4(\mathbb{D}^2)}^4 \lesssim_{\eps}   \lambda^{1/2} \delta^{1/2} \ln(\lambda).\]

\myindent In the following, using that, for all $n$, there is at most one $k_n \geq 0$ such that $(k_n,n)\in \mathcal{Z}$ (see \cite{chabert2025infty}), we drop the index $k$ in the notations, and, in particular, we write $Z$ the projection on $\mathcal{Z}$ on the second coordinate.

\section{Reduction to a problem of oscillatory integrals}\label{secreduc}

\myindent In this section, we present how to further reduce the problem to a number of quantitative estimates of coefficient given by explicit integrals in terms of Bessel functions. They indexed by a particular type of 4-uplets of $Z^4$ satisfying a crucial \textit{resonant condition}, which we introduce in Section \ref{subsectrapezes}. Then, we decompose the coefficient with respect to different regimes, depending on the local behavior of the Bessel functions (exponentially small, large and non oscillatory, of order 1 and oscillatory) in Section \ref{subsecdefosc}. Finally, in the \textit{oscillatory} regimes, we reduce the estimates of the coefficient to quantitative estimates on explicit oscillatory integrals with a large parameter $\lambda$.

\subsection{Integrals indexed by zero sum sets}\label{subsectrapezes}

\myindent Observe that
\[\|u\|_{L^4(\mathbb{D}^2)}^4 = \sum_{n_1,n_2,n_3,n_4 \in Z} a_{n_1} \overline{a_{n_2} a_{n_3}} a_{n_4} \left(\Pi_{i=1}^4 c_{n_i} \right) \times \int_0^1 r dr   \Pi_{i=1}^4 \lambda_{n_i}^{1/2} J_{n_i}(\lambda_{n_i} r) \int_{\mathbb{T}} e^{i(n_1 - n_2 - n_3 + n_4) \theta } d\theta,\]

where $\mathbb{T} := \R/(2\pi \Z)$, the constants $c_{n_i}$ are all of order $1$, and $\lambda_{n_i} = \lambda + O(\delta)$. In particular, it is natural to introduce the following coefficients, which govern the analysis.

\begin{definition}
    For $\textbf{n} = (n_1,n_2,n_3,n_4)\in Z^4$, we introduce
    \[\begin{split}
        C_{\textbf{n}} &:= \frac{1}{2\pi \Pi_{i=1}^4 c_{n_i}} \int_{\mathbb{D}^2} dx  \phi_{n_1}(x) \overline{\phi_{n_2}(x) \phi_{n_3}(x)} \phi_{n_4}(x) \\
        &=\int_0^1 r dr  \lambda_{n_1}^{1/2} J_{n_1} (\lambda_{n_1} r) \lambda_{n_2}^{1/2} J_{n_2} (\lambda_{n_2} r) \lambda_{n_3}^{1/2} J_{n_3} (\lambda_{n_3} r) \lambda_{n_4}^{1/2} J_{n_4} (\lambda_{n_4} r) \fint_{\mathbb{T}} e^{i(n_1 - n_2 - n_3 + n_4) \theta} d\theta.
    \end{split}\]
\end{definition}

\myindent Now, we see that the $\theta$ integral vanishes whenever $n_1 - n_2 - n_3 + n_4 \neq 0$, and equals $1$ when $n_1 - n_2 - n_3 + n_4 = 0$. This can be seen a \textit{resonance condition} on the quadruplet $\textbf{n}$, made explicit thanks to the separation of variables. In the following, we may restrict the sum to those $\textbf{n}$ such that $n_1 + n_4 = n_2 + n_3$, which we call \textit{zero sum sets}. In the following, we will only use estimates on $|C_{\textbf{n}}|$, hence we may without loss of generality assume moreover that 
\[n_1 \leq n_2 \leq n_3 \leq n_4.\]

\myindent Indeed, if we introduce the set of zero sum sets (the reason for the notation $T$, for \textit{trapeze}, will be introduced in Section \ref{subsecuncertainty})
\[T := \{(n_1,n_2,n_3,n_4) \in Z^4 \qquad n_1 \leq n_2 \leq n_3 \leq n_4, \ n_1 + n_4 = n_2 + n_3\}, \]
we will prove that
\[\sum_{\textbf{n} \in T} |a_{n_1}||a_{n_2}| |a_{n_3}||a_{n_4}| |C_{\textbf{n}}| \lesssim_\eps \lambda^{1/2} \delta^{1/2}\ln(\lambda).\]

\myindent Finally, for $\textbf{n} \in T$, we define its \textit{side}, \textit{top} and \textit{base} by
\[\begin{split}
    &s(\textbf{n}) := n_2 - n_1 = n_4 - n_3 \\
    &t(\textbf{n}) := n_3 - n_2 \\
    &b(\textbf{n}) := n_4 - n_1
\end{split}\]

\subsection{Definition of oscillatory and non oscillatory regimes and enunciation of the bounds}\label{subsecdefosc}

\myindent From standard BKW analysis, the function $\lambda_n^{1/2} J_n(\lambda_n r)$ is oscillatory in the classically allowed region $r > \mu_n$, and exponentially decaying in the classically forbidden region $r < \mu_n$, whereas its behavior near the caustic is given by the Airy function. More precisely, for $\mu_n$ away from zero, one expects that (see \cite{guillemin1973remarks})

\[\lambda_n^{1/2} J_n(\lambda_n r) \simeq \lambda^{1/6} Ai(-\lambda_n^{2/3} \rho(r)),\]
with $\rho(r) \simeq r - \mu_n$. We thus see that, for $n\in Z$, and $r\in [0,1]$, there are essentially three different regimes for $\lambda_n^{1/2} J_n(\lambda_n r)$:

\myindent i. If $r$ is sufficiently smaller than $\mu_n$, then it is exponentially small.

\myindent ii. If $r$ is close to $\mu_n$, it is of size $\lambda_n^{1/6}$ and not very oscillatory.

\myindent iii. If $r$ is sufficiently larger than $\mu_n$, it is of size $1$ and it oscillates.

\myindent Moreover, those three behavior are well separated if we choose the $r-\mu_n$ thresholds for this trisection at $r-\mu_n = -\lambda^{-2/3 + \kappa}$ and $r- \mu_n = \lambda^{-2/3 + \kappa}$ for some $\kappa > 0$. Now, we insist that this is only a heuristic, since, for example, when $\mu_n$ is very close to zero, $\lambda_n^{1/2} J_n(\lambda_n \mu_n)$ is much larger than $\lambda_n^{1/6}$ (indeed, $\lambda^{1/2}J_0(0)$ is of order $\lambda^{1/2}$). However, we may quantify the three different behaviors. Regarding the exponential decay below the caustic point, we prove superpolynomial decay in Lemma \ref{expdecay}. Near the caustic, the main consequence of the Airy approximation is that, upon differentiating $k$ times $r\mapsto\lambda_n^{1/2} J_n(\lambda_n r)$, if $|r-\mu_n|$ is small enough, one loses only a factor $\lambda^{2/3 k}$ instead of a factor $\lambda_n^k$ as a naive differentiation seems to indicate. We quantify this phenomenon in Lemma \ref{derlemm}. Finally, we give a very precise estimate of the oscillatory behavior above the caustic in Lemma \ref{Lemmstatphase}.

\myindent This motivates the quantitative trisection of $C_{\textbf{n}}$ into three different regimes.

\begin{definition}\label{deftrisec}
    Let $\kappa > 0$. For $\textbf{n} \in T$, find 
    \[1 = \eta_{n_4}^{int}(r) + \eta_{n_4}^{caus}(r) + \eta_{n_4}^{ext}(r)\]
    a partition of unity of $[0,1]$ by smooth localizers $\eta_{n_4}^{\bullet}(r)$ (i.e. $0 \leq \eta_{n_4}^{\bullet}(r) \leq 1$) such that $\eta_{n_4}^{int}$ vanishes for $r \geq \mu_{n_4} - \frac{1}{2}\lambda^{-2/3 + \kappa}$, $\eta_{n_4}^{caus}$ vanishes for $r\notin (\mu_{n_4} \pm \lambda^{-2/3 + \kappa})$, and $\eta_{n_4}^{ext}$ vanishes when $r \leq \mu_{n_4} + \frac{1}{2}\lambda^{-2/3 + \kappa}$. Finally, we assume that 
    \[\left|\frac{d^k}{dr^k} \eta_{n_4}^{\bullet}(r)\right| \lesssim_k \lambda^{k(2/3 - \kappa)},\]
    where we stress that the implicit constant is independent of $\textbf{n}$.
    
    We define the trisection
    \[C_{\textbf{n}} = C_\textbf{n}^{int} + C_\textbf{n}^{caus} + C_\textbf{n}^{ext},\]
    where
    \[C_\textbf{n}^{\bullet} = \int_0^1 r dr \eta_{n_4}^{\bullet}(r)  \lambda_{n_1}^{1/2} J_{n_1} (\lambda_{n_1} r) \lambda_{n_2}^{1/2} J_{n_2} (\lambda_{n_2} r) \lambda_{n_3}^{1/2} J_{n_3} (\lambda_{n_3} r) \lambda_{n_4}^{1/2} J_{n_4} (\lambda_{n_4} r).\]
\end{definition}

\myindent Now, in the cases when two or more of the caustics $(\mu_{n_i})$ are too close, then some additional resonances and cancellations may appear in the oscillations of the $\lambda_{n_i}^{1/2} J_{n_i}(\lambda_{n_i} r)$, leading to exceptional growth (think of the integral of a $\cos^2(\lambda r)$). Hence, we filter moreover the set of zero sum sets to remove these types of phenomena. In particular, we again discriminate by comparing the distance between caustics to $\lambda^{-2/3}$, which is equivalent to comparing the distance between the $(n_i)$ to $\lambda^{1/3}$. Similarly to the definition of the trisection (Definition \ref{deftrisec}), we allow a $\lambda^{\alpha}$ loss, for some small $\alpha > 0$.

\begin{definition}
    Let $\alpha > 0$. We define the set of $\alpha$-\textit{degenerate} zero sum sets by
    \[T_{\alpha} := \{\textbf{n} \in T \qquad \text{such that}\ \min(s(\textbf{n}), t(\textbf{n})) \leq \lambda^{1/3 + \alpha}\}.\]
\end{definition}

\myindent We now introduce the general bounds which we will prove in the following.

\begin{proposition}\label{defregimes}
    Let $\alpha > 0$, and let $\kappa > 0$ such that $\kappa < \alpha/100$. Then, there holds
    
    \myindent i. For $\alpha$-degenerate zero sum sets,
    \[\sum_{\textbf{n} \in T_{\alpha}} |a_{n_1}||a_{n_2}| |a_{n_3}||a_{n_4}| |C_\textbf{n}| \lesssim_\eps \lambda^{1/3 + \alpha} \ln(\lambda).\]
    
    \myindent ii. More generally, for zero sum sets such that $\min(s(\textbf{n}), t(\textbf{n})) \leq \lambda^{1/2} \delta^{1/2}$, there holds
    \[\sum_{\min(s(\textbf{n}), t(\textbf{n})) \leq \lambda^{1/2} \delta^{1/2}} |a_{n_1}||a_{n_2}| |a_{n_3}||a_{n_4}| |C_\textbf{n}| \lesssim_\eps \lambda^{1/2}\delta^{1/2} \ln(\lambda).\]
    
    \myindent iii. For the interior estimate,
    \[\sum_{\textbf{n} \in T} |a_{n_1}||a_{n_2}| |a_{n_3}||a_{n_4}| |C_\textbf{n}^{int}| \lesssim_{\eps,\kappa,N} \lambda^{-N}.\]
    
    \myindent iv. For the caustic estimate, 
    \[\sum_{\textbf{n} \in T\backslash T_\alpha} |a_{n_1}||a_{n_2}| |a_{n_3}||a_{n_4}| |C_\textbf{n}^{caus}| \lesssim_{\eps,\kappa,N} \lambda^{-N}.\]
    
    \myindent v. Finally, for the exterior estimate, there holds
    \[\sum_{\textbf{n} \in T\backslash T_\alpha, \min(s(\textbf{n}), t(\textbf{n})) \geq \lambda^{1/2}\delta^{1/2}} |a_{n_1}||a_{n_2}| |a_{n_3}||a_{n_4}| |C_\textbf{n}^{ext}| \lesssim_{\eps, \kappa} \lambda^{1/2} \delta^{1/2}\ln(\lambda)^{1/2}.\]
\end{proposition}

\myindent In particular, Theorem \ref{main} follows from the proposition. We give the proof of the first three estimates, which are non-oscillatory. The rest of this paper is devoted to proving the last two estimates.

\begin{proof}
    From the a priori estimate on the $L^4$ norm of $\lambda_n^{1/2} J_n(\lambda_n r)$ given by Lemma \ref{l4norm}, and applying Hölder inequality, observe first that
    \begin{equation}\label{univbound}
        \forall \textbf{n} \in T \qquad |C_\textbf{n}|\lesssim \ln(\lambda).
    \end{equation}

    \myindent Thanks to the bound \eqref{univbound}, we may prove the first two estimates of the proposition. Indeed, let us fix more generally $M > 0$, and let us write
    \[\begin{split}
        \sum_{s(\textbf{n}) \leq M} |a_{n_1}||a_{n_2}||a_{n_3}||a_{n_4}| |C_{\textbf{n}}| &\lesssim \ln(\lambda) \sum_{1 \leq s \leq M} \sum_{\textbf{n} \in T, \ s(\textbf{n}) = s} |a_{n_1}||a_{n_1+ s}| |a_{n_4 -s}||a_{n_4}| \\
        &\lesssim \ln(\lambda) \sum_{1 \leq s\leq M} \left(\sum_{n_1 \in Z} |a_{n_1}||a_{n_1 + s}| \right)\left(\sum_{n_4 \in Z} |a_{n_4-s}||a_{n_4}| \right)\\
        &\lesssim M \ln(\lambda),
    \end{split}\]
    where we apply the Cauchy-Schwarz inequality with respect to $n_1$ (resp $n_4$) to obtain that the sum on $n_1$ (resp $n_4$) is bounded by $1$ in the last inequality. 
    
    \myindent Similarly, we may write
    \[\begin{split}
        \sum_{t(\textbf{n}) \leq M} |a_{n_1}||a_{n_2}||a_{n_3}||a_{n_4}| |C_{\textbf{n}}| &\lesssim \ln (\lambda) \sum_{1\leq t\leq M} \sum_{\textbf{n} \in T, t(\textbf{n}) = t} |a_{n_1}||a_{n_2}||a_{n_2 + t}||a_{2n_2 + t - n_1}| \\
        & \lesssim \ln(\lambda) \sum_{1 \leq t \leq M} \sum_{n_2 \in Z} |a_{n_2}||a_{n_2 + t}| \sum_{n_1 \in Z} |a_{n_1}||a_{2n_2 + t - n_1}|\\
        &\lesssim \ln(\lambda) \sum_{1 \leq t \leq M} \sum_{n_2 \in Z} |a_{n_2}||a_{n_2 + t}|\\
        &\lesssim M \ln(\lambda),
    \end{split}\]
    where we apply the Cauchy-Schwarz inequality with respect to $n_1$, for a fixed $n_2,t$, then with respect to $n_1$ for a fixed $t$.
    
    \myindent Now, apply this to $M = \lambda^{1/3 + \alpha}$ or $M = \lambda^{1/2}\delta^{1/2}$.
    
    \quad
    
    \myindent With regards to the third estimate in the proposition, observe that is it enough to prove that, for any $\textbf{n} \in T$, and any $N \geq 1$, there holds
    \[|C_{\textbf{n}}^{int}| \lesssim_{\eps,\kappa,N} \lambda^{-N}.\]
    
    \myindent Now, observe that, on the support of $\eta_{n_4}^{int}$, for any $N\geq 1$, Lemma \ref{expdecay} yields that
    \[\lambda_{n_4}^{1/2} |J_{n_4}(\lambda_{n_4} r)| \lesssim_N \lambda^{-N},\]
    where the upper bound is independent on $n_4$ and on $r$. Moreover, all the other quantities are bounded by a fixed power of $\lambda$.
\end{proof}

\subsection{Reducing the oscillatory cases to oscillatory integrals}\label{subsecredosc}

\myindent In order to estimate the coefficients $C_{\textbf{n}}^{caus}$ and $C_{\textbf{n}}^{int}$, we start by making more explicit their structure as oscillatory integrals. First, we give the following technical lemma, which quantifies the oscillatory structure of $\lambda_n^{1/2} J_n(\lambda_n r)$ for $r - \mu_n$ large enough.

\begin{lemma}\label{Lemmstatphase}
	For any $n\in Z$, for any $N\geq 0$ and for any $r - \mu_n \gtrsim \lambda^{-2/3 + \kappa}$, there holds
	\begin{equation}\label{statphaseforJn}
	\left|\lambda_n^{1/2} J_n(\lambda_n r) - \frac{\cos(\lambda_n(\mu_n \alpha_n - r \sin(\alpha_n)))}{(r^2 - \mu_n^2)^{1/4}} a_{n,N,\lambda}(r) \right| \lesssim_{\kappa,N} \lambda^{-N},
	\end{equation}
	where $\alpha_n(r)$ is defined by
	\[\alpha_n(r) := \arccos(\mu_n / r),\]
	and where $a_{n,N,\lambda}(r)$ is a smooth function such that
	\begin{equation}\label{deranearcaus}
	    \left|\frac{d^K}{dr^K} a_{n,N,\lambda}(r)\right| \lesssim_{\eps,\kappa, N} (r-\mu_n)^{-K}.
	\end{equation}
\end{lemma}

\myindent We leave the proof of the lemma to Appendix \ref{applemm1} since it is quite technical. Indeed, while, for a fixed $r$, a stationary phase expansion from the definition of $J_n(\lambda_n r)$ (see \eqref{defJn}) immediately yields the asymptotic estimate \eqref{statphaseforJn} for $\lambda$ large, the subtlety is to control precisely the function $a_{n,N,\lambda}(r)$ (and its derivatives), and, in particular, to obtain implicit constants in the estimates which are independent of $r$.  As an example of this subtlety, we remark that we need to use the stationary phase expansion with significantly more terms than $N$ in order to have a uniform $O(\lambda^{-N})$ remainder, and that the estimate stops working at $r-\mu_n = \lambda^{-2/3}$.

\myindent Thanks to this lemma, and the observation that, if $N$ is large enough, then, for some fixed constant $C$,
\[\sum_{\textbf{n} \in T} |a_{n_1}||a_{n_2}||a_{n_3}||a_{n_4}| O(\lambda^{-N}) = O(\lambda^{C - N}),\]
we see that, in order to prove the estimates of Proposition \ref{defregimes}, we may equivalently replace the coefficients $C_{\textbf{n}}^{caus}$, $C_{\textbf{n}}^{ext}$ by their \textit{oscillatory} parts $c_{\textbf{n}}^{caus}$, $c_{\textbf{n}}^{caus}$, which we define as follows.

\begin{definition}
	For $\textbf{n} \in T$, and $N$ large enough fixed, we introduce the coefficients
	\[\begin{split}
	&c_\textbf{n}^{ext} := \int_0^1 rdr \eta_{n_4}^{ext} \Pi_{i=1}^4 \cos\left(\lambda_{n_i}\left(\mu_{n_i} \alpha_{n_i}(r) - r\sin(\alpha_{n_i}(r))\right) \right) a_{n_i,N}(r) \\
	&c_\textbf{n}^{caus} := \int_0^1 rdr \eta_{n_4}^{caus} \Pi_{i=1}^3 \cos\left(\lambda_{n_i}\left(\mu_{n_i} \alpha_{n_i}(r) - r\sin(\alpha_{n_i}(r))\right)\right) a_{n_i,N}(r) \lambda_{n_4}^{1/2} J_{n_4} (\lambda_{n_4} r) .
	\end{split}\]
\end{definition}

\myindent In particular, if we introduce the function
\[f(r,\mu) := \mu \arccos(\mu / r) - \sqrt{r^2 - \mu^2},\]
observe that we may rewrite $c_\textbf{n}^{ext}$ as a sum of integrals of the form
\[I_{\pm,\pm,\pm}(\textbf{n}) := \int  \exp i \left(\lambda_{n_1}f(r,\mu_{n_1}) \pm\lambda_{n_2}f(r,\mu_{n_2}) \pm \lambda_{n_3}f(r,\mu_{n_3})\pm \lambda_{n_4}f(r,\mu_{n_4})\right) \times \Pi_{i=1}^4 a_{n_i,N}(r) r dr \eta_{n_4}^{ext}\]
and their complex conjugates. Now, it is more convenient to write $I_{\pm,\pm,\pm}$ as an oscillatory integral with a single large parameter $\lambda$. Since $\lambda_{n_i} = \lambda + O(\delta)$, we write
\begin{equation}\label{deferrorterm}
    \lambda_{n_1}f(r,\mu_{n_1}) \pm\lambda_{n_2}f(r,\mu_{n_2}) \pm \lambda_{n_3}f(r,\mu_{n_3})\pm \lambda_{n_4}f(r,\mu_{n_4}) = \lambda \Phi_{\pm,\pm,\pm} (r, \boldsymbol{\nu}) + e(\lambda,r,\boldsymbol{\nu}),
\end{equation}
where we define
\[\Phi_{\pm,\pm,\pm} (r, \boldsymbol{\nu}) := f(r,\nu_1) \pm f(r,\nu_2) \pm f(r,\nu_3) \pm f(r,\nu_4),\]
where 
\[\nu := \frac{n}{\lambda},\]
and $e(\lambda,r,\boldsymbol{\nu})$ is an error term. Observe that
\[|\nu - \mu_n| = n\left|\frac{1}{\lambda} - \frac{1}{\lambda_n}\right| = n \frac{|\lambda-\lambda_n|}{\lambda \lambda_n} \lesssim \delta\lambda^{-1}.\]

\myindent The error term is uniformly small, in the sense of the following result.
\begin{lemma}\label{controlerror}
	For all $\lambda \geq 1$, for all $\mathbf{n}\in T$, define, for $r - \mu_{n_4} \gtrsim \lambda^{-2/3 + \kappa}$,
	\[a_{err}(r, \mathbf{n}, \lambda) := e^{i e(\lambda,r,\boldsymbol{\nu})},\]
	where the error term $e$ is defined by equation \eqref{deferrorterm}. Then, there holds
	\[ \left| \frac{d^K}{dr^K} a_{err}(r,\mathbf{n},\lambda)\right| \lesssim_{\kappa} (r-\mu_{n_4})^{-K}.\]
\end{lemma}

\begin{proof}
    One needs only prove the bound
\[\left|\frac{d^k}{dr^k} e(\lambda,r,\boldsymbol{\nu})\right|\lesssim_{\kappa} (r-\mu_{n_4})^{-k}.\]
\myindent Now, for any $i$,
\[\left|\frac{d^k}{dr^k} \left(\lambda_{n_i} f(r,\mu_{n_i}) - \lambda f(r,\nu_i)\right)\right| \lesssim |\lambda - \lambda_{n_i}| \left|\frac{d^k}{dr^k} f(r,\mu_{n_i})\right| + \lambda \left|\frac{d^k}{dr^k} f(r,\mu_{n_i}) - \frac{d^k}{dr^k} f(r,\nu_i)\right|.\]

\myindent From the explicit expression of $f$, one finds that, for any $r \gtrsim \nu + \lambda^{-2/3 +\kappa}$,
\[\left|\frac{d^k}{dr^k} f(r,\nu)\right| \lesssim (r-\nu)^{3/2 - k},\]
and
\[\left|\frac{d}{d\nu} \frac{d^k}{dr^k} f(r,\nu)\right| \lesssim (r-\nu)^{1/2 - k}.\] 

\myindent In particular, we may further bound
\[\left|\frac{d^k}{dr^k} f(r,\mu_{n_i}) - \frac{d^k}{dr^k} f(r,\nu_i)\right| \lesssim \lambda |\mu_{n_i} - \nu_i| (r-\mu_{n_i})^{1/2 - k} \lesssim (r-\mu_{n_i})^{1/2-k}.\]
\end{proof}

\myindent Thanks to this approximation, we may finally write $c_{\mathbf{n}}^{ext}$ as a standard oscillatory integral. Indeed, fixing $N\geq 1$ large enough, set
\[a_{\mathbf{n}}^{ext}(r) := r \Pi_{i=1}^4 a_{n_i,N}(r)  a_{err}(r,\mathbf{n},\lambda).\]

\myindent Then, we may write
\[I_{\pm,\pm,\pm}(\textbf{n}) = \int dr \eta_{n_4}^{ext} (r) e^{i \lambda \Phi_{\pm,\pm,\pm}(r,\boldsymbol{\nu})} a_{\mathbf{n}}^{ext}(r),\]
where the symbol $a_{\mathbf{n}}^{ext}(r)$ enjoys the bound
\[\left| \frac{d^K}{dr^K} a_{\mathbf{n}}^{ext}(r) \right| \lesssim_{\kappa} (r - \mu_{n_4})^{-K},\]
or, equivalently, since $|\mu_{n_4} - \nu_4| = O(\delta \lambda^{-1})$,
\[\left| \frac{d^K}{dr^K} a_{\mathbf{n}}^{ext}(r) \right| \lesssim_{\kappa} (r - \nu_4)^{-K}.\]

\quad

\myindent With a similar proof, we may also rewrite $c_{\mathbf{n}}^{caus}$ as a sum of oscillatory integrals of the form
\[I_{\pm,\pm}(\textbf{n}) := \int dr \eta_{n_4}^{caus}(r) e^{i\lambda \Phi_{\pm,\pm}(r,\boldsymbol{\nu}) }a_{\mathbf{n}}^{caus}(r) \lambda_{n_4}^{1/2} J_{n_4}(\lambda_{n_4} r),\]
where 
\[\Phi_{\pm,\pm}(r,\boldsymbol{\nu}) := f(r,\nu_1) \pm f(r,\nu_2) \pm f(r,\nu_3),\]
and the symbol $a_{\mathbf{n}}^{caus}(r)$ enjoys the bound
\begin{equation}\label{boundderacaus}
    \left| \frac{d^K}{dr^K} a_{\mathbf{n}}^{caus}(r) \right| \lesssim_{\kappa} (r - \nu_3)^{-K}.
\end{equation}

\section{Analysis of the leading term : convexity and improved spacing}\label{secmain}

\myindent In this section, we prove the bound
\begin{equation}\label{boundmain}
    \sum_{\mathbf{n} \in T\backslash T_\alpha, \min(s(\mathbf{n}),t(\mathbf{n})) \geq \lambda^{1/2} \delta^{1/2} } |a_{n_1}||a_{n_2}| |a_{n_3}||a_{n_4}| |I_{-,-,+}(\textbf{n})| \lesssim_{\eps,\kappa} \lambda^{1/2} \delta^{1/2}\ln(\lambda)^{1/2}.
\end{equation}

\myindent As we will argue in Section \ref{secremaind}, the case $(-,-,+)$ is the leading term in the analysis. Indeed, observe that, since we reduced the problem to an estimate of a number of oscillatory integrals with a large parameter $\lambda$, the analysis depends primarily on the gradient of the phases which appear, i.e. on $\partial_r \Phi_{\pm,\pm,\pm}(r,\boldsymbol{\nu})$ and $\partial_r \Phi_{\pm,\pm}(r, \boldsymbol{\nu})$. Now, these are linear combinations of $\partial_r f(r,\nu_i)$, and, as we will see, the monotonicity of $\nu \mapsto \partial_r f(r,\nu)$ yields straightforward lower bounds on the norm of $\partial_r \Phi_{\pm,\pm,\pm}(r,\boldsymbol{\nu})$ and $\partial_r \Phi_{\pm,\pm}(r, \boldsymbol{\nu})$ in all cases, except for the exterior case $(-,-,+)$. In order to study this case, one needs to use the convexity of $\nu \mapsto \partial_r f(r,\nu)$ in a nontrivial way.

\myindent In Section \ref{subsectrapavg}, we will thus prove the estimate, for all $\textbf{n}  \in T\backslash T_\alpha, \min(s(\mathbf{n}),t(\mathbf{n})) \geq \lambda^{1/2} \delta^{1/2}$,
\begin{equation}\label{boundI--+}
    \left| I_{-,-,+}(\textbf{n})\right| \lesssim_{\eps,\alpha} \frac{\lambda}{s(\textbf{n}) b(\textbf{(n)})},
\end{equation}
which is essentially optimal, since one can prove that $ I_{-,-,+}(\textbf{n}) \simeq \left(\frac{\lambda}{s(\textbf{n}) b(\textbf{(n)})}\right)^2$, and the sum of the latest quantity doesn't seem to enjoy better bounds in general. 

\myindent Then, in Section \ref{subsecuncertainty}, we will explain how to sum estimate \eqref{boundI--+} to obtain the bound \eqref{boundmain}, that is we will prove that
\begin{equation}\label{summation}
    \sum_{\mathbf{n} \in T\backslash T_\alpha, \min(s(\mathbf{n}),t(\mathbf{n})) \geq \lambda^{1/2} \delta^{1/2} } |a_{n_1}||a_{n_2}| |a_{n_3}||a_{n_4}|\frac{\lambda}{s(\textbf{n}) b(\textbf{(n)})} \lesssim_{\eps,\kappa} \lambda^{1/2} \delta^{1/2}\ln(\lambda)^{1/2}
\end{equation}

\myindent This relies on the introduction of a new counting estimate on the distribution of zero sum sets $\textbf{n} \in T$, which can be seen as a \textit{quadrilateral} version of the standard \textit{Jarnik's Lemma} on triangles inscribed inside a convex curve, see Proposition \ref{uncertainty}.

\subsection{Second order average of convex function and estimate of $I_{-,-,+}(\textbf{n})$}\label{subsectrapavg}

\myindent Before giving the proof of estimate \eqref{boundI--+}, we introduce a straightforward result, which we choose to write as a lemma given its importance in the following.

\myindent Consider $g$ a strictly convex function on an interval $I \subset \R$, and find $\nu_1,\nu_2,\nu_3,\nu_4 \in I$ satisfying the zero sum set condition $\nu_1 + \nu_4 = \nu_2 + \nu_3$, and, without loss of generality $\nu_1 < \nu_2 < \nu_3 < \nu_4$. We consider the following \textit{second order average} of $g$ along the zero sum set $\boldsymbol{\nu} = (\nu_1,\nu_2,\nu_3,\nu_4)$
\[M_g(\boldsymbol{\nu}) := g(\nu_1) - g(\nu_2) - g(\nu_3) + g(\nu_4),\]
which, modulo a factor $1/2$, can be interpreted as the difference between the average of $g(\nu_1)$ and $g(\nu_4)$, and the average of $g(\nu_2)$ and $g(\nu_3)$. Observe that, due to the zero sum set condition $\nu_1 + \nu_4 = \nu_2 + \nu_3$, then $M_g(\boldsymbol{\nu})$ vanishes in the case where $g$ is linear, and it is positive when $g$ is strictly convex, hence it is naturally a second order quantity which measures the convexity of $g$ on the zero sum set $\boldsymbol{\nu}$. We give the following quantitative version of that statement.

\begin{lemma}\label{convavg}
    There holds for any function $g$ (not necessarily convex)
    \[M_g(\boldsymbol{\nu}) = \int_{\nu_1}^{\nu_4} g''(\nu) \left((\nu - \nu_1) 1_{\nu_1 \leq \nu \leq \nu_2} + s(\boldsymbol{\nu}) 1_{\nu_2 \leq \nu \leq \nu_3} + (\nu_4 - \nu) 1_{\nu_3 \leq \nu \leq \nu_4}\right) d\nu.\]
    
    \myindent In particular when $g$ is strictly convex, there holds
    \[\frac{1}{2} s(\boldsymbol{\nu}) b(\boldsymbol{\nu}) \min_{\nu \in [\nu_1,\nu_4]} g''(\nu) \leq M_g(\boldsymbol{\nu}) \leq s(\boldsymbol{\nu}) b(\boldsymbol{\nu}) \max_{\nu \in [\nu_1,\nu_4]} g''(\nu).\]
\end{lemma}

\begin{proof}
    Write
    \[\begin{split}
        M_g(\boldsymbol{\nu}) &= - \int_{\nu_1}^{\nu_2} g'(\nu) d\nu + \int_{\nu_3}^{\nu_4} g'(\nu) d\nu  \\
        &= \int_0^{s(\boldsymbol{\nu})} (g'(\nu_3 + \nu) - g'(\nu_2 - \nu)) d\nu \\
        &= \int_0^{s(\boldsymbol{\nu})} \int_{\nu_2-\nu}^{\nu_3 +\nu} g''(u) du d\nu \\
        &= \int_{\nu_1}^{\nu_4} g''(\nu)\left((\nu - \nu_1) 1_{\nu_1 \leq \nu \leq \nu_2} + s(\boldsymbol{\nu}) 1_{\nu_2 \leq \nu \leq \nu_3} + (\nu_4 - \nu) 1_{\nu_3 \leq \nu \leq \nu_4}\right) d\nu.
    \end{split}\]
\end{proof}

\myindent Thanks to this lemma, we can now prove estimate \eqref{boundI--+}. We recall that $I_{-,-,+}(\textbf{n})$ is written as an oscillatory integral
\[I_{-,-,+}(\textbf{n}) = \int dr \eta_{n_4}^{ext}(r) e^{i\lambda \Phi_{-,-,+}(r,\boldsymbol{\nu})} a_{\textbf{n}}^{ext}(r).\]
\begin{proof}
    For $r \in [0,1]$ and $\nu \in [0,r]$, let 
    \[g(r)(\nu) := \partial_r f(r,\nu) = - \frac{1}{r}\sqrt{r^2 - \nu^2},\]
    which is a strictly convex function of $\nu$ for a fixed $r$. Indeed, let us write, to fix ideas,
    \[\begin{split}
        &\partial_\nu\partial_r f(r,\nu) = \frac{\nu}{r\sqrt{r^2 - \nu^2}}\\
        &\partial_{\nu\nu}\partial_r f(r,\nu) = \frac{r}{(r^2 - \nu^2)^{3/2}}.
    \end{split}\]
    
    \myindent Then, observe that 
    \[\partial_r \Phi_{-,-,+}(r,\boldsymbol{\nu}) = M_{g(r)}(\boldsymbol{\nu}),\]
    and thus, thanks to Lemma \ref{convavg}, we deduce that $\partial_r \Phi_{-,-,+}$ doesn't vanish on the support of $\eta_{n_4}^{ext}(r)$. In particular, we may integrate by parts $K$ times in $r$, and write that 
    \[\left|I_{-,-,+}(\textbf{n})\right| \leq \lambda^{-K} \int \left|J^K\left(\eta_{n_4}^{ext} a_{\textbf{n}}^{ext} \right) (r)\right| dr + B_K(\boldsymbol{\nu}),\]
    where we define the vector field
    \[Ju := \frac{d}{dr} \left(\frac{1}{\partial_r \Phi_{-,-,+}(\cdot,\boldsymbol{\nu})} u\right),\]
    and where $B_K(\boldsymbol{\nu})$ is a boundary term which appears since the integrand doesn't vanish at $r = 1$, namely 
    \[B_K(\boldsymbol{\nu})= \sum_{k=1}^K \lambda^{-k} \frac{1}{\partial_r \Phi_{-,-,+}(1, \boldsymbol{\nu})} \left|J^{k-1} \left(\eta_{n_4}^{ext} a_{\textbf{n}}^{ext} \right) (1) \right|.\]
    
    \myindent Now, we claim that there holds
    \begin{equation}\label{estJK}
        \left|J^K\left(\eta_{n_4}^{ext} a_{\textbf{n}}^{ext}\right)(r)\right| \lesssim\frac{1}{(\partial_r \Phi_{-,-,+}(r,\boldsymbol{\nu}))^K} \frac{1}{(r-\nu_4)^{K+1}}.
    \end{equation}
    
    \myindent Indeed, observe that, when iterating $J$, each time a derivative falls on $a$ or $\eta_{n_4}^{ext}$, one loses at most a factor $(\partial_r \Phi_{-,-,+}(r,\boldsymbol{\nu}))^{-1} (r-\nu_4)^{-1}$. Moreover, each time the derivative falls on terms depending on $\Phi_{-,-,+}$, the loss can be bounded be the same factor once one observes that 
    \[\left|\frac{\partial_r^p\Phi_{-,-,+}}{\partial_r \Phi_{-,-,+}} \right|\lesssim (r-\nu_4)^{-(p-1)}.\]
    \myindent This comes from the observation that 
    \[\partial_r^p \Phi_{-,-,+} (r,\boldsymbol{\nu}) = M_{g^{(p-1)}(r)}(\boldsymbol{\nu}),\]
    which yields the result from the explicit expression of second order averages, and the fact that, from a direct computation,
    \[|\partial_r^p \partial_\nu^2 f(r,\nu)| \lesssim (r-\nu_4)^{-(p-1)} \partial_r \partial_\nu^2 f(r,\nu).\]
    
    \myindent In particular, since $1 - \nu_4 \geq \eps$, we may already bound the boundary term $B_K(\boldsymbol{\nu})$ by
    \[B_K(\boldsymbol{\nu}) \lesssim_{\eps, K} \sum_{k=1}^K \lambda^{-k} \frac{1}{M_{g(1)}(\boldsymbol{\nu})^k},\]
    which, thanks to Lemma \ref{convavg}, can again be bounded by
    \[B_K(\boldsymbol{\nu}) \lesssim_{\eps,K} \sum_{k=1}^K \frac{1}{(\lambda s(\boldsymbol{\nu}) b(\boldsymbol{\nu}))^k} = \sum_{k=1}^K \left(\frac{\lambda}{s(\textbf{n}) b(\textbf{n})}\right)^k.\]
    
    \myindent Now, the quadrilateral Jarnik's Lemma which we will prove in Proposition \ref{uncertainty} ensures that, since $s(\textbf{n}) \gtrsim \lambda^{1/2} \delta^{1/2}$, and hence $s(\textbf{n})b(\textbf{n}) \gtrsim s(\textbf{n})^2 \gtrsim \lambda \delta$, then there holds
    \[s(\textbf{n}) b(\textbf{n}) \gtrsim \lambda.\]
    \myindent In particular, we may finally write that
    \[B_K(\boldsymbol{\nu}) \lesssim_{\eps,K} \frac{\lambda}{s(\textbf{n}) b(\textbf{n})}.\]
    
    \quad
    
    \myindent Now, we need to prove the estimate
    \[\lambda^{-K} \int_{r \gtrsim \nu_4 + \lambda^{-2/3 + \kappa}} \frac{1}{M_{g(r)}(\boldsymbol{\nu})^K} \frac{1}{(r-\nu_4)^{K+1}} dr \lesssim \frac{\lambda}{s(\textbf{n}) b(\textbf{n})},\]
    provided $K$ is large enough. We decompose the integral into different zones.
    
    \myindent i. For $\lambda^{-2/3 + \kappa} \leq r-\nu_4 \leq s(\boldsymbol{\nu})$, from the explicit expression of $M_{g(r)}(\boldsymbol{\nu})$ given in Lemma \ref{convavg}, and the computation of $\partial_r\partial_\nu^2 f(r,\nu)$, we may bound
    \[M_{g(r)}(\boldsymbol{\nu}) \gtrsim s(\boldsymbol{\nu})^2 (r-\nu_3)^{-3/2} \gtrsim s(\boldsymbol{\nu})^{1/2},\]
    hence there holds
    \[\lambda M_{g(r)}(\boldsymbol{\nu}) (r-\nu_4) \gtrsim \lambda s(\boldsymbol{\nu})^{1/2} \lambda^{-2/3 + \kappa} \gtrsim \lambda^{\alpha/2 + \kappa},\]
    thus this part of the integral is $O(\lambda^{-N})$ for any $N$ as long as $K$ is large enough.
    
    \myindent ii. For $s(\boldsymbol{\nu}) \leq r - \nu_4 \leq \frac{1}{10} b(\boldsymbol{\nu})$, we may bound
    \[M_{g(r)}(\boldsymbol{\nu}) \gtrsim s(\boldsymbol{\nu})(\partial_r \partial_\nu f(r,\nu_3) - \partial_r \partial_\nu f(r,\nu_2)) \gtrsim s(\boldsymbol{\nu}) \partial_r \partial_\nu f(r,\nu_3) \gtrsim s(\boldsymbol{\nu}) (r-\nu_3)^{-1/2}.\]
    
    Indeed, the point is that 
    \[\partial_r \partial_\nu f(r,\nu_2) = \frac{\nu_2}{r(r + \nu_2)^{1/2}(r-\nu_2)^{1/2}} \leq \frac{\nu_3}{r(r+\nu_3)^{1/2}} (r-\nu_2)^{-1/2} \leq \frac{1}{\sqrt{2}} \frac{\nu_3}{r(r+\nu_3)^{1/2}} (r-\nu_3)^{-1/2}= \frac{1}{\sqrt{2}}\partial_r \partial_\nu f(r,\nu_3),\]
    as long as $\frac{\nu_3 - \nu_2}{r - \nu_2} \geq 1$ (for example). In particular, we deduce that 
    \[\lambda M_{g(r)}(\boldsymbol{\nu}) \gtrsim \lambda s(\boldsymbol{\nu}) (r-\nu_3)^{-1/2} (r-\nu_4) \gtrsim \lambda s(\boldsymbol{\nu})^{3/2} \gtrsim \lambda^{3\alpha / 2},\]
    hence that part of the integral is a $O(\lambda^{-N})$ for any $N$ as long as $K$ is large enough.
    
    \myindent iii. Finally, if $r - \nu_4 \gtrsim \frac{1}{10}b(\boldsymbol{\nu})$, we may now write that $r-\nu_4 \simeq r- \nu_1$, hence
    \[\lambda M_{g(r)} (\boldsymbol{\nu}) (r-\nu_4) \gtrsim \lambda s(\boldsymbol{\nu}) b(\boldsymbol{\nu}) (r-\nu_1)^{-3/2} (r-\nu_4) \gtrsim \lambda s(\boldsymbol{\nu}) b(\boldsymbol{\nu}),\]
    hence that part of the integral is a $O\left(\left(\frac{\lambda}{s(\textbf{n}) b(\textbf{n})}\right)\right)$, thus a $O\left(\frac{\lambda}{s(\textbf{n}) b(\textbf{n})}\right)$ as above.
\end{proof}

\subsection{Quadrilateral Jarnik's Lemma and counting estimate}\label{subsecuncertainty}

\myindent In this section, we start by a crucial proposition on the spacing of the four points in a zero sum set, and, as a direct corollary, on the spacing between two zero sum sets with same top, which we then use to prove the counting estimate \eqref{summation}. 

\begin{proposition}\label{uncertainty}
    Let $\textbf{n} \in T$. If
    \[s(\textbf{n}) b(\textbf{n}) \gtrsim \lambda \delta, \]
    then there actually holds
    \[s(\textbf{n}) b(\textbf{n}) \gtrsim \lambda.\]
    
    \myindent As a corollary, the following spacing version holds : let $n_1', n_4' \in Z$ such that $\textbf{n}' = (n_1', n_2, n_3, n_4') \in T$, and such that $\sigma := s(\textbf{n}') - s(\textbf{n}) > 0$. Then, if 
    \[\sigma b(\textbf{n}) \gtrsim \lambda \delta,\]
    there holds 
    \[\sigma b(\textbf{n}) \gtrsim \lambda.\]
\end{proposition}

\myindent Before turning to the proof of the proposition, we give a brief visual interpretation of it in terms of the \textit{area} of a \textit{trapezoidal quadrilateral} inscribed into a convex curve, and moreover with vertices in a lattice. Indeed, find $\gamma$ a strictly convex curve in $\R^2$, and consider the lattice $\Lambda := \frac{1}{\lambda}\Z^2$. A classical result, commonly known as Jarnik's Lemma (see \cite{jarnik1926gitterpunkte}), yields that, if $p_1,p_2,p_3$ are three distinct points of $\Lambda$ which lie on $\gamma$, then two among the three are at a distance at least $c\lambda^{-2/3}$ for some $c(\gamma) > 0$. This is a direct consequence of the fact that the three points form a triangle inscribed inside $\gamma$, which is necessarily non flat due to the convexity of $\gamma$. Then, one uses a lower bound on the area of the triangle using moreover that its vertices lie on the lattice $\Lambda$.

\myindent Now, find instead four distinct points $p_1,p_2,p_3,p_4$ in $\Lambda \cap \gamma$, which are at a distance $\ll 1$ one from another, and consider the inscribed quadrilateral $Q \subset \gamma$ that they form. Assume moreover that, in suitable coordinates locally, $\gamma = \{y = g(x)\}$ is the graph of a strictly convex function $g$, and that, if, in those coordinates, $p_i = (\nu_i, g(\nu_i))$, with $\nu_1 < \nu_2 < \nu_3 < \nu_4$, then the quadruplet $\boldsymbol{\nu}$ satisfies the zero sum set condition $\nu_1 + \nu_4 = \nu_2 + \nu_3$. Then, the area of $Q$ is, on the one hand, bounded from below by a constant depending only on $g$ times $s(\boldsymbol{\nu}) b(\boldsymbol{\nu})$. On the other hand, the upper bound and lower bound of second order averages of strictly convex functions on a zero sum $\boldsymbol{\nu}$ set by constants times its area $s(\boldsymbol{\nu}) b(\boldsymbol{\nu})$ in Lemma \ref{convavg} thus yield that the area of $Q$ is of the order of the second order average $M_g(\boldsymbol{\nu})$, which is a weighted sum of points in $\frac{1}{\lambda}\Z$. In particular, since $Q$ is non flat due to the convexity of $\gamma$, this sum is nonzero, and thus at least $1/\lambda$. Overall, one finds a bound of the form
\[s(\boldsymbol{\nu})b(\boldsymbol{\nu}) \gtrsim \lambda^{-1},\]
which immediately yields, for example, that $|p_1 - p_2|, |p_3 - p_4| \gtrsim \lambda^{-1/2}$. The proof of the proposition follows the same ideas, however, since we have a $\delta > 0$ size of spectral window, we need moreover to consider quadrilateral which are inscribed in a small $O(\delta \lambda^{-1})$ neighborhood of the curve $\gamma$. Thus, we finally prove a conditional result, where the condition ensures that the quadrilateral $Q$ is \textit{non flat}, i.e. we need to assume a priori some distance between the $p_i$ to guarantee that the $\delta \lambda^{-1}$-neighborhood of $\gamma$ is still "curved".

\myindent Finally, we note that, if the points $p_i$ are already very close one to another (say $\lambda^{-\eps}$), then the above condition yields that the quadrilateral $Q$ is, to the main order, a \textit{trapeze}, since, modulo a remainder of a smaller order in $\lambda$, the condition of zero sum set yields that the sides $(p_1,p_4)$ and $(p_2,p_3)$ are parallel. In fact, if $\gamma = \{y = f(x) =  Ax^2 + Bx + C\}$ is an arc of parabola (i.e. it has constant curvature), one can compute that the zero sum set condition on the coordinates $(x_i)$ is equivalent to the fact that the points $(x_i, f(x_i))$ form an inscribed trapeze.

\begin{proof}
    First, assume that we can prove that, for $i=1,2,3,4$, there exists an integer $m_i \in \Z$ such that
    \begin{equation}\label{lattice}
        f(1,\nu_i) = \frac{m_i + 1/2}{\lambda} \pi + O(\delta \lambda^{-1}).
    \end{equation}
    
    \myindent Then, fix 
    \[g(\nu) = -f(1,\nu) = \sqrt{1 - \nu^2} - \nu \arccos(\nu),\]
    which is a strictly convex function on $[\eps,1 - \eps]$ (it is a primitive of $-\arccos(\nu)$). We may thus write
    \[s(\boldsymbol{\nu})b(\boldsymbol{\nu}) \gtrsim M_g(\boldsymbol{\nu}) = \frac{m_2 + m_3 - m_1 - m_4}{\lambda} \pi + O(\delta \lambda^{-1}).\]
    
    \myindent Now, if we assume that $s(\textbf{n}) b(\textbf{n}) \geq C\lambda \delta$ for $C> 0$ a large enough constant which doesn't depend on $\textbf{n}$ or $\lambda,\delta$, the left-hand side is strictly larger than the remainder in the right-hand side of the equation. In particular, $m_2 + m_3 - m_1 - m_4$ cannot vanish. Since it is an integer, it is thus larger than $1$, and we may finally write
    \[s(\boldsymbol{\nu}) b(\boldsymbol{\nu}) \gtrsim \lambda^{-1},\]
    which is equivalent to the result.
    
    \myindent We now turn to the proof of \eqref{lattice}. Let us fix, more generally, any $n \in Z$. Applying Lemma \ref{Lemmstatphase} for $N = 1$ and $r = 1$, we find that, since $J_n(\lambda_n) = 0$ by definition,
    \[\frac{\cos(\lambda_n f(1, \mu_n))}{(1 - \mu_n^2)^{1/4}} a_{n,1}(1,\mu_n) = O(\lambda^{-1}),\]
    which we may rewrite, since $1 - \mu_n \geq \eps$, as
    \[\cos(\lambda_n f(1, \mu_n)) = O(\lambda^{-1}).\]
    
    \myindent In particular, this ensures that there exists $m \in \Z$ such that
    \[\lambda_n f(1, \mu_n) = (m + 1/2) \pi + O(\lambda^{-1}).\]
    
    \myindent Now, letting $\nu = n/\lambda$, we may rewrite that
    \[\lambda_n f(1, \mu_n) = (\lambda+O(\delta))f(1, \nu + O(\delta \lambda^{-1})) = \lambda f(1,\nu) + O(\delta),\]
    and thus, dividing the equality by $\lambda$,
    \[f(1,\nu) = \frac{m+1/2}{\lambda} \pi + O(\delta \lambda^{-1}).\]
\end{proof}

\myindent We now prove the counting estimate \eqref{summation} using the improved spacing of zero sum sets. 

\begin{proof}
    Let us first write
    \[\begin{split}
        \sum_{s(\textbf{n}) \gtrsim \lambda^{1/2} \delta^{1/2}} |a_{n_1}||a_{n_2}||a_{n_3}||a_{n_4}| \frac{\lambda}{s(\textbf{n}) b(\textbf{n})} &\leq \left(\sum_{\textbf{n} \in T} |a_{n_1}|^2 |a_{n_3}|^2 |a_{n_4}|^2\right)^{1/2} \left(\sum_{\textbf{n}\in T, s(\textbf{n}) \gtrsim \lambda^{1/2} \delta^{1/2}} |a_{n_2}|^2 \frac{\lambda^2}{s(\textbf{n})^2 b(\textbf{n})^2}\right)^{1/2} \\
        &\leq \left(\sum_{\textbf{n}\in T, s(\textbf{n}) \gtrsim \lambda^{1/2} \delta^{1/2}} |a_{n_2}|^2 \frac{\lambda^2}{s(\textbf{n})^2 b(\textbf{n})^2}\right)^{1/2}.
    \end{split}\]
    
    \myindent Now, let us first compute the sum when $n_2,n_3 \in Z$ are fixed. In particular, the top of any zero sum set with second and third coordinates $n_2$, $n_3$ is fixed and equals $t := n_3 - n_2$. Moreover, each such zero sum set is determined entirely by its side $s\gtrsim \lambda^{1/2} \delta^{1/2}$, and its base is $2s + t$. Hence, we need to compute
    \[\sum_{s \gtrsim \lambda^{1/2} \delta^{1/2}} \frac{\lambda^2}{s^2 (2s + t)^2} 1_{(n_2 - s,n_2,n_3,n_3 + s)\in T}.\]
    
    \myindent The interest is that we may now use the improved spacing for zero sum sets given in Proposition \ref{uncertainty}. Indeed, let us enumerate those zero sum sets $\textbf{n} \in T$ with fixed second and third coordinates $n_2,n_3$ by increasing side, say $s(\textbf{n}^{(1)}) < s(\textbf{n}^{(2)}) <...$. Applying Proposition \ref{uncertainty}, observe that, if 
    \[s(\textbf{n}^{(i+k)}) - s(\textbf{n}^{(i)}) \gtrsim \frac{\lambda \delta}{b(\textbf{n}^{(i)})},\]
    then there holds
    \[s(\textbf{n}^{(i+k)}) - s(\textbf{n}^{(i)}) \gtrsim \frac{\lambda}{b(\textbf{n}^{(i)})}.\]
    
    \myindent Thanks to this spacing fact, we may inductively partitionate the $\textbf{n}^{(i)}$ into clusters, in the following sense : find a sequence $1 = i_1 < i_2 < ...$, and denote $I_l := \{i_l \leq i < i_{l+1}\}$. Then, we may impose
    
    \myindent i. Inside a cluster $i\in I_l$, there holds 
    \[s(\textbf{n}^{(i)}) - s(\textbf{n}^{(i_l)}) \lesssim \frac{\lambda \delta}{b(\textbf{n}^{(i_l)})}.\]
    \myindent In particular, the size of the cluster $I_l$ is bounded by $\max\left(\frac{\lambda \delta}{b(\textbf{n}^{(i_l)})}, 1\right)$.
    
    \myindent ii. Moreover, between two clusters,
    \[s(\textbf{n})^{(i_{l+1})} - s(\textbf{n}^{(i_l)}) \gtrsim \frac{\lambda}{b(\textbf{n}^{(i_l)})}.\]
    
    \myindent We claim that, as a consequence, for $l \geq 1$, and for any $i\in I_l$, there holds
    \[s(\textbf{n}^{(i)}) b(\textbf{n}^{(i)}) \gtrsim l\lambda.\]
    
    \myindent Indeed, let us prove this fact inductively. For $l = 1$, since we assume that $s(\textbf{n}^{(1)}) \gtrsim \lambda^{1/2} \delta^{1/2}$, then there holds
    \[b(\textbf{n}^{(1)} )s(\textbf{n}^{(1)} ) \geq s(\textbf{n}^{(1)} )^2 \gtrsim \lambda \delta.\]
    Thus, Proposition \ref{uncertainty} yields that $s(\textbf{n}^{(1)}) b(\textbf{n}^{(1)} ) \gtrsim \lambda$.
    
    \myindent Now, if this is true for some $l\geq 1$, write
    \[\begin{split} s(\textbf{n}^{(i_{l+1})}) b(\textbf{n}^{(i_{l+1})}) &\geq (\textbf{n}^{(i_{l+1})})b(\textbf{n}^{(i_l)})\\   &=(s(\textbf{n})^{(i_{l+1})} - s(\textbf{n}^{(i_l)})) b(\textbf{n}^{(i_l)}) +s(\textbf{n}^{(i_l)}) b(\textbf{n}^{(i_l)}) \\
    &\gtrsim \lambda + l \lambda = (l+1) \lambda.
    \end{split}\]
    
    \myindent As a conclusion, and, since for any $l$ there holds $b(\textbf{n}^{(i_l)}) \geq t$, we may write
    \[\begin{split}
        \sum_{s \gtrsim \lambda^{1/2} \delta^{1/2}} \frac{1}{s^2 (2s + t)^2} 1_{(n_2 - s,n_2,n_3,n_3 + s)\in T} &\leq \sum_{l\geq 1} \frac{\lambda^2}{(s(\textbf{n}^{(i_l)}) b(\textbf{n}^{(i_l)}))^2} |I_l| \\
        &\lesssim \sum_{l\geq 1} \frac{1}{l^2} |I_l|\\
        &\lesssim \sup_l |I_l|\\
        &\lesssim \max\left(\frac{\lambda \delta}{t}, 1\right)
    \end{split}\]
    
    \myindent Overall, summing finally on $n_2$ and $t$ yields
    
    \[\begin{split}
        \left(\sum_{\textbf{n}\in T, s(\textbf{n}) \gtrsim \lambda^{1/2} \delta^{1/2}} |a_{n_2}|^2 \frac{\lambda^2}{s(\textbf{n})^2 b(\textbf{n})^2}\right)^{1/2} &= \left(\sum_{n_2 \in Z} \sum_{t\geq 1} \sum_{s \gtrsim \lambda^{1/2} \delta^{1/2}} \frac{\lambda^2}{s^2 (2s + t)^2} 1_{(n_2 - s,n_2,n_2 + t,n_2 + t+ s)\in T}\right)^{1/2} \\
        &\lesssim \left(\sum_{n_2 \in Z} |a_{n_2}|^2\sum_{\lambda \geq t\geq 1} \left(\frac{\lambda \delta}{t} + 1_{n_2 + t \in Z}\right)\right)^{1/2}\\
        &\lesssim  \left(\sum_{n_2 \in Z} |a_{n_2}|^2\left(\lambda\delta \ln(\lambda) +  \#Z\right)\right)^{1/2}\\
        &\lesssim(\lambda \delta \ln(\lambda))^{1/2},
    \end{split}\]
    since the cardinal of $Z$ is bounded by $O(\lambda \delta)$ thanks to the spectral estimate \eqref{improvedremainder}. This concludes the proof.
\end{proof}

\section{Estimate of the remainder terms}\label{secremaind}

\myindent In this section, we prove that all remaining terms are error terms, in the following sense: first, in the exterior case, we prove $I_{-,-,+}(\textbf{n})$ is the leading term, as was claimed in Section \ref{secmain}. Precisely, we will only prove that
\begin{equation}\label{estremainder}
    \forall (\pm,\pm,\pm) \neq (-,-,+) \qquad |I_{\pm,\pm,\pm}(\textbf{n})| \lesssim \frac{\lambda}{s(\textbf{n}) b(\textbf{n})},
\end{equation}
which is enough to conclude thanks to the counting estimate \eqref{summation}. Second, we prove that, in the caustic case, there holds 
\begin{equation}\label{estcaus}
    \forall (\pm,\pm) \qquad I_{\pm,\pm}(\textbf{n}) = O(\lambda^{-\infty})
\end{equation} 
as long as $\lambda \notin T_\alpha$, which implies the interior estimate in Proposition \ref{defregimes}.

\subsection{Estimate of $I_{\pm,\pm,\pm}( \textbf{n})$}\label{subsecext}

\myindent In this section, we prove estimate \eqref{estremainder} for any $\textbf{n}$ such that $\min(s(\textbf{n}), t(\textbf{n})) \gtrsim \lambda^{1/2} \delta^{1/2}$. We recall that we need to estimate the oscillatory integral
\[I_{\pm,\pm,\pm}(\textbf{n}) := \int dr \eta_{n_4}^{ext}(r) e^{i\lambda \Phi_{\pm,\pm,\pm}(r,\boldsymbol{\nu})} a_{\textbf{n}}^{ext}(r),\]
where
\[\Phi_{\pm,\pm,\pm}(r,\boldsymbol{\nu}) = f(r,\nu_1) \pm f(r,\nu_2) \pm f(r,\nu_3) \pm f(r,\nu_4).\]

\myindent As for the case $(-,-,+)$, the analysis is essentially two steps: first, proving lower bounds on $|\partial_r \Phi_{\pm,\pm,\pm}|$ using the properties of $\nu \mapsto \partial_r f(r,\nu)$. Second, integrate by parts in $r$ a large number of times and control the resulting integral, possibly by cutting it into different pieces, and a boundary term.

\subsubsection{Easy monotonicity cases : $(+,+,+),(+,+,-),(+,-,+),(-,+,+),(-,-,-)$}\label{subsubseceasy}

\myindent Thanks to the monotonicity and nonpositivity of $\nu \mapsto \partial_r f(r,\nu)$,  we may already prove the following lower bound: in the cases $(+,+,+),(+,+,-),(+,-,+),(-,+,+),(-,-,-)$, and on the support of $\eta_{n_4}^{ext}$, there holds
\begin{equation}\label{boundpartPhiremaind}
    |\partial_r \Phi(r,\boldsymbol{\nu})| \gtrsim (r-\nu_4)^{-1/2}.
\end{equation}

\myindent Moreover, for higher order derivatives, in any cases, there holds
\[\left|\partial_r^p\Phi_{-,-,+} \right|\lesssim r^{-1}(r-\nu_4)^{3/2-p}.\]

\begin{proof}
    The second inequality is straightforward. Regarding the first, observe that
    
    \myindent i. In the case $(+,+,+)$, $\partial_r \Phi_{+,+,+} \Phi$ is the sum of four terms of the same sign, so its absolute value is larger than the absolute value of any of them, for example $|\partial_r f(r,\nu_4)|$, itself larger than $(r-\nu_4)^{1/2}$.
    
    \myindent ii. In the case $(-,+,+)$ (resp $(+,-,+)$), $\partial_r f(r,\nu_1) - \partial_r f(r,\nu_2) < 0$ (resp $\partial_r f(r,\nu_1) - \partial_r f(r,\nu_3)$ since $\nu \mapsto \partial_r f(r,\nu)$ is increasing. Hence, since it is moreover negative, we find a lower bound by $|\partial_r f(r,\nu_4)|$.
    
    \myindent iii. In the case $(+,+,-)$, the same reasoning leads (for example) to a lower bound $|\partial_r f(r,\nu_2)|$, which is itself larger than $|\partial_r f(r,\nu_4)|$.
    
    \myindent iv. Finally, in the case $(-,-,-)$, one should write
    \[\begin{split}
        \partial_r f(r,\nu_1) - \partial_r f(r,\nu_2) - \partial_r f(r,\nu_3) - \partial_r f(r,\nu_4) &= \left[\partial_r f(r,\nu_1) - \partial_r f(r,\nu_2) - \partial_r f(r,\nu_3) + \partial_r f(r,\nu_4)\right] - 2 \partial_r f(r,\nu_4) \\
        &\geq -2 \partial_r f(r,\nu_4), 
    \end{split}\]
    where we use the positivity of the second order average, see Lemma \ref{convavg}, hence we again find the same lower bound.
\end{proof}

\myindent Now, we may apply the gradient estimate to integrate by parts $K$ times in $I_{\pm,\pm,\pm}(\textbf{n})$. Indeed, using the vector field
\[J u := \frac{d}{dr} \left(\frac{1}{\partial_r \Phi_{\pm,\pm,\pm}(\cdot, \boldsymbol{\nu})} u \right),\]
we find that
    \[\left|I_{\pm,\pm,\pm}(\textbf{n})\right| \leq \lambda^{-K} \int \left|J^K\left(\eta_{n_4}^{ext} a_{\textbf{n}}^{ext} \right) (r)\right| dr + B_K(\boldsymbol{\nu}),\]
    where $B_K(\boldsymbol{\nu})$ is the boundary term 
    \[B_K(\boldsymbol{\nu})= \sum_{k=1}^K \lambda^{-k} \frac{1}{\partial_r \Phi_{\pm,\pm,\pm}(1, \boldsymbol{\nu})} \left|J^{k-1} \left(\eta_{n_4}^{ext} a_{\textbf{n}}^{ext} \right) (1) \right|.\]
    
    \myindent Now, we claim that, with a similar proof than for estimate \eqref{estJK}, there holds, for some constant $C >0$,
    \[\left| J^K \left(\eta_{n_4}^{ext} a_{\textbf{n}}\right)(r)\right| \lesssim \frac{1}{|\partial_r \Phi_{\pm,\pm,\pm}(r,\boldsymbol{\nu})|^K (r-\nu_4)^{K + C}}.\]
    
    \myindent In particular, the upper bound is itself bounded by $(r-\nu_4)^{-C - 3/2 K}$ thanks to estimate \eqref{boundpartPhiremaind}. Hence, on the one hand, for $r = 1$, since $1-\nu_4 \geq \eps,$ we have the straightforward bound
    \[B_K(\boldsymbol{\nu}) \lesssim_\eps \lambda^{-1},\]
    whereas, on the other hand, for $r < 1$, there holds for another constant $C'$,
    \[\lambda^{-K}\left|J^K\left(\eta_{n_4}^{ext} a_{\textbf{n}}^{ext} \right) (r)\right| \lesssim (r-\nu_4)^{-C - 3/2 K} \lesssim \lambda^{-3/2 K \kappa + C'},\]
    which is a $O(\lambda^{-N})$ for all $N$ as long as $K$ is large enough. Overall, we have proved that
    \[|I_{\pm,\pm,\pm}(\textbf{n})| \lesssim_{\eps,\alpha} \lambda^{-1},\]
    which itself is bounded trivially by $\frac{\lambda}{s(\textbf{n}) b(\textbf{n})}$.
    
\subsubsection{Harder monotonicity cases : $(-,+,-), (+,-,-)$}\label{subsubsechard}

\myindent In the cases $(-,+,-)$ and $(+,-,-)$, the monotonicity of $\nu \mapsto \partial_r f(r,\nu)$ yields a priori that $\partial_r \Phi$ is negative, but no quantitative lower bound on its absolute value. Instead of computing exactly such quantitative lower bounds, we will rather prove that $\partial_r \Phi$, and higher order derivatives, satisfy at least the same estimate than than in Section \ref{subsectrapavg}. Hence, applying the same reasoning will yield at least the same bound. We thus need only prove two estimates, that is
\[|\partial_r \Phi(r,\boldsymbol{\nu})| \geq M_{g(r)} (\boldsymbol{\nu}),\]
and 
 \[\left|\frac{\partial_r^p\Phi} {\partial_r \Phi} \right|\lesssim (r-\nu_4)^{-(p-1)}.\]

\myindent For the first estimate, write
\[\begin{split}
    |\partial_r \Phi_{-,+,-} (r,\boldsymbol{\nu})| &= \partial_r f(r,\nu_2) - \partial_r f(r,\nu_1) + \partial_r f(r,\nu_4) - \partial_r f(r,\nu_3) \\
    &= \left[\partial_r f(r,\nu_4) - \partial_r f(r,\nu_3) - (\partial_r f(r,\nu_2) - \partial_r f(r,\nu_1)\right] + 2(\partial_r f(r,\nu_2) - \partial_r f(r,\nu_1)) \\
    &\geq \left[\partial_r f(r,\nu_4) - \partial_r f(r,\nu_3) - (\partial_r f(r,\nu_2) - \partial_r f(r,\nu_1)\right]\\
    &= M_{g(r)} (\boldsymbol{\nu}) \\
    |\partial_r \Phi_{+,-,-}| (r,\boldsymbol{\nu}) &= \partial_r f(r,\nu_4) - \partial_r f(r,\nu_1) + \partial_r f(r,\nu_3) - \partial_r f(r,\nu_2) \\
    &= M_{g(r)} (\boldsymbol{\nu}) + 2( \partial_r f(r,\nu_3) - \partial_r  f(r,\nu_1)) \\
    &\geq M_{g(r)} (\boldsymbol{\nu}).
\end{split}\]

\myindent For the second estimate, for example in the case $(-,+,-)$, write
\[\begin{split}
    \left|\partial_r^p \Phi_{-,+,-} (r,\boldsymbol{\nu}) \right| &= \left| \partial_r^p f(r,\nu_1) - \partial_r^p f(r,\nu_2) + \partial_r^p f(r,\nu_3) - \partial_r^p f(r,\nu_4) \right|\\
    &\leq \int_{\nu_1}^{\nu_2} |\partial_\nu \partial_r^p f(r,\nu)| d\nu + \int_{\nu_3}^{\nu_4}  |\partial_\nu \partial_r^p f(r,\nu)| d\nu\\
    &\leq (r-\nu_4)^{-(p-1)} \left(\int_{\nu_1}^{\nu_2} |\partial_\nu \partial_r f(r,\nu)| d\nu + \int_{\nu_3}^{\nu_4}  |\partial_\nu \partial_r f(r,\nu)| d\nu\right) \\
    &= (r-\nu_4)^{-(p-1)} |\partial_r \Phi_{-,+,-}(r,\boldsymbol{\nu})|.
\end{split}\]

\myindent In particular, the argument of Section \ref{subsectrapavg} applies and we find
\[|I_{-,+,-} (\textbf{n})|, |I_{+,-,-}(\textbf{n})| \lesssim \frac{\lambda}{s(\textbf{n}) b(\textbf{n})}.\]

\subsection{Estimate of $C_{\textbf{n}}^{caus}$}\label{subseccaus}

\myindent In this section, we prove estimate \eqref{estcaus} for any $\textbf{n} \in T \backslash T_\alpha$. Recall that we need to estimate the oscillatory integral 

\[I_{\pm,\pm}(\textbf{n}) := \int dr \eta_{n_4}^{caus}(r) e^{i\lambda \Phi_{\pm,\pm}(r,\boldsymbol{\nu}) }a_{\mathbf{n}}^{caus}(r) \lambda_{n_4}^{1/2} J_{n_4}(\lambda_{n_4} r),\]

	 \myindent In order to estimate $I_{\pm,\pm}$, we start by proving lower bounds on $\partial_r \Phi_{\pm,\pm}(r,\boldsymbol{\nu})$. Precisely, we claim that there holds
	 \begin{equation}\label{localboundcaus}
	     \left|\partial_r \Phi_{\pm,\pm}(r,\boldsymbol{\nu})\right| \gtrsim  r^{-1/2}\lambda^{-1/3 + 1/2 \kappa}.
	 \end{equation}
	 
	 \myindent Indeed, in the cases $(+,+),(+,-), (-,+)$, this easily follows from the fact that $\nu \mapsto \partial_r f(r,\nu)$ is increasing and negative, and the fact that, for $i=1,2,3$, and $r$ in the support of $\eta_{n_4}^{caus}$,
	 \[|\partial_r f(r,\nu_i)| \gtrsim r^{-1/2}(r-\nu_i)^{1/2} \gtrsim r^{-1/2}\lambda^{-1/3 + 1/2\alpha}.\]
	 
	 \myindent In the case $(-,-)$, one needs to use moreover convexity properties. First, for any $r$ in the support of $\eta_{n_4}^{caus}$, there holds $|r-\nu_4| \lesssim \lambda^{-2/3 + \kappa}$. Moreover, we may bound $|\partial_\nu \partial_r f(r,\nu_i)|\lesssim s(\boldsymbol{\nu})^{-1/2}$. Hence, for $i = 1,2,3$, there holds
	 \[\left|f(r,\nu_i) - f(\nu_4,\nu_i)\right| \lesssim s(\boldsymbol{\nu})^{-1/2}\lambda^{-2/3 +\kappa} \lesssim \lambda^{-1/3 + \kappa - \alpha/2}.\]
	 
	 \myindent In particular, if we define the strictly convex function
	 \[g(\nu) := f(\nu_4,\nu),\]
	 which vanishes at $\nu = \nu_4$, we may write
	 \[\partial_r \Phi_{-,-}(r,\boldsymbol{\nu}) = M_g(\boldsymbol{\nu}) + O(\lambda^{-1/3 + \kappa - \alpha/2}).\]
	 \myindent Now, from the explicit expression of second order averages given by Lemma \ref{convavg}, we may bound
	 \[M_g(\boldsymbol{\nu}) \geq \frac{1}{2} s(\boldsymbol{\nu})^2 \inf_{\nu \in (\nu_3,\nu_4)} g''(\nu).\]
	 
	 \myindent Using the explicit expression of $\partial_\nu^2\partial_r f(r,\nu)$, we may further bound
	 \[M_g(\boldsymbol{\nu}) \gtrsim r^{-1/2}s(\boldsymbol{\nu})^{2 - 3/2} \gtrsim r^{-1/2}\lambda^{-1/3 +\alpha/2}.\]
	 \myindent To conclude, observe that
	 \[\frac{\lambda^{-1/3 + \alpha/2}}{\lambda^{-1/3 + \kappa - \alpha/2}} = \lambda^{\alpha - \kappa} >> 1,\]
	 since we assumed that $\alpha > \kappa$. In particular, this proves the bound in the case $(-,-)$.
	 
	 \quad
	 
	 \myindent Now, thanks to the bound \eqref{localboundcaus}, we may integrate by parts $K$ times n $r$ on the support of $\eta_{n_4}^{caus}(r)$, and find that
	 \[|I_{\pm,\pm}| \leq \lambda^{-K} \int \left|J^K\left(\eta_{n_4}^{caus} a \lambda^{1/2} J_{n_4}(\lambda_{n_4} \cdot)\right)(r) \right| dr,\]
	 where $J$ is defined by
	 \[Ju := \frac{d}{dr}\left(\frac{1}{\partial_r \Phi_{\pm,\pm}(\cdot, \boldsymbol{\nu})} u\right).\]
	 
	 \myindent Now, observe that, for $r$ in the support of $\eta_{n_4}^{caus}$, we may write, for all $k\geq 0$, using the bound of Lemma \ref{derlemm},
	 \[\left| \frac{d^k}{dr^k} \left(\lambda^{1/2} J_{n_4}(\lambda_{n_4} \cdot )\right)(r) \right| \lesssim \lambda^{C + 2/3 k + 1/2 k\kappa} r^{-k/2}.\]
	 
	 \myindent Using moreover the bounds \eqref{boundderacaus} on the derivatives of $\eta_{n_4}^{caus}$ and $a$, there holds, for some constant $C > 0$, for all $k\geq 0$,
	 \[\left|\frac{d^k}{dr^k}\left(\eta_{n_4}^{caus} a \lambda^{1/2} J_{n_4}(\lambda_{n_4} \cdot)\right)(r)\right| \lesssim \lambda^{C +2/3 k + k\kappa} r^{-K/2},\]
	 from which we easily deduce that, using \eqref{localboundcaus}, for some constant $C > 0$, there holds
	 \[\left|J^K\left(\eta_{n_4}^{caus} a \lambda^{1/2} J_{n_4}(\lambda_{n_4} \cdot)\right)(r) \right| \lesssim \lambda^{K + C - K(1/2\alpha - \kappa)}.\]
	 
	\myindent In particular, choosing $\kappa <1/2 \alpha$, there holds
	 \[\lambda^{-K} \int \left|J^K\left(\eta_{n_4}^{caus} a \lambda^{1/2} J_{n_4}(\lambda_{n_4} \cdot)\right)(r) \right| dr = O(\lambda^{-\infty}) \qquad K \to \infty.\]

\appendix 

\section{Estimates of Bessel functions}\label{AppBessel}
\subsection{$L^4$ norm of joint eigenfunctions}\label{Appnorm}

\myindent In this appendix, we prove the following uniform bound on the $L^4$ norm of joint eigenfunctions.

\begin{lemma}\label{l4norm}
    For all $n\in \Z$, for all $\lambda \geq 1$, there holds
    \[\left\| \lambda^{1/2} J_n(\lambda |\cdot|)\right\|_{L^4(\mathbb{D}^2)}\lesssim \ln(\lambda)^{1/4},\]
    where the upper bound is independent of $n$.
\end{lemma}

\begin{proof}
    Let us start by writing that
    \begin{equation}\label{eqappendixnorm}
        \left\| \lambda^{1/2} J_n(\lambda \cdot)\right\|_{L^4(\mathbb{D}^2)}^4 = 2\pi \lambda^2 \int_0^1 |J_n(\lambda r)|^4 r dr = 2\pi \int_0^\lambda |J_n(r)|^4 r dr.
    \end{equation}
    \myindent Now, using \cite{guo1997uniform}[Lemma 3.2], there holds, for all $p > 4$, and for all $n$,
    \[\left(\int_0^{\infty} |J_n(r)|^p r dr\right)^{1/p} \leq C(p-4)^{-1/p},\]
    where the constant $C > 0$ is independent of $p$ and of $n$. In particular, applying Hölder inequality to \eqref{eqappendixnorm}, we may write that, for any $\eps > 0$,
    \[\begin{split}
        \int_0^\lambda |J_n(r)|^4 rdr &\leq \left(\int_0^\lambda |J_n(r)|^{4(1+\eps)} r dr \right)^{\frac{1}{1+\eps}} \left(\lambda^2/2\right)^{\frac{\eps}{1 + \eps}} \\
        &\leq 2^{-\frac{\eps}{1+\eps}} \left(C (4\eps)^{-\frac{1}{4(1+\eps)}}\right)^{4} \lambda^{\frac{2\eps}{1+\eps}}\\
        &\leq 2^{-\frac{\eps}{1+\eps}} C^4 4^{-\frac{1}{1+\eps}} \eps^{-\frac{1}{1+\eps}} \lambda^{\frac{2\eps}{1+\eps}}\\
        &\leq C' \eps^{-\frac{1}{1+\eps}} \lambda^{\frac{2\eps}{1+\eps}}\\
        &\leq C' \eps^{-1} \lambda^{2\eps},
    \end{split} \]
    with $C' = C^4$. Now, the function $\eps \mapsto -\ln \eps + 2\eps \ln \lambda$ reaches its minimum at $\eps = \frac{1}{2 \ln \lambda}$. For this value of $\eps$, the upper bound reads
    \[\int_0^\lambda |J_n(r)|^4 r dr \leq 2 e C'\ln \lambda,\]
    from which the claim follows.
\end{proof}

\subsection{Decay of $J_n(\lambda \cdot)$ in the interior region}\label{Appintdecay}

\myindent In this section, we prove the following lemma.

\begin{lemma}\label{expdecay}
    For all $n \in Z$ and for all $r\in [0,1]$ such that $\mu_n - r \gtrsim \lambda^{-2/3 + \kappa}$, there holds, for any $N\geq 0$
    \[|\lambda_n^{1/2} J_n(\lambda_n r) | \lesssim_{\kappa,N} \lambda^{-N}.\]
\end{lemma}

\begin{proof}
     For $r = 0$, observe that, for the condition $\mu_n - r > 0$ to be non empty, there must hold $n \neq 0$, and hence $J_n(0) = 0$. Now, assume that $r \neq 0$ and write
    \[J_n(\lambda_n r) = \int_{\mathbb{T}} e^{i\lambda_n r \left(\frac{\mu_n}{r}\alpha - \sin \alpha\right)} d\alpha.\]
    \myindent Now, the phase function 
    \[\phi : \alpha \mapsto \frac{\mu_n}{r} \alpha - \sin \alpha\]
    satisfies
    \[\phi'(\alpha) = \frac{\mu_n}{r} - \cos \alpha = \frac{\mu_n - r}{r} + 1 - \cos \alpha \geq \frac{\mu_n - r}{r} > 0\]
    and all its higher order derivatives are uniformly bounded by $1$. Moreover, we claim that there holds 
    \[\left|\frac{\phi''(\alpha)}{\phi'(\alpha)}\right| \lesssim \max(1,\left(\frac{\mu_n - r}{r}\right)^{-1/2}).\]
    \myindent Indeed, assume that $\eps := \frac{\mu_n - r}{r} \leq 1$. Then, for $\alpha \lesssim \eps^{1/2}$, one needs only write that
    \[\left|\frac{\phi''(\alpha)}{\phi'(\alpha)}\right| = \frac{\sin \alpha}{\phi'(\alpha)} \lesssim \frac{\eps^{1/2}}{\eps}.\]
    \myindent For $\alpha \gtrsim \eps^{1/2}$, one may instead use that 
    \[\left|\frac{\phi''(\alpha)}{\phi'(\alpha)}\right| \leq \left|\frac{\sin \alpha}{1 - \cos \alpha}\right| \lesssim \alpha^{-1} \lesssim \eps^{-1/2}.\]
    \myindent Thus, thanks to the above estimate, we see that, after integrating by parts $K$ times in $\alpha$, one may bound
    \[\begin{split}
        |J_n(\lambda_n r)| &\lesssim (\lambda_n r)^{-K} \left(\frac{\mu_n-r}{r}\right)^{-K} \left(\max\left(1,\frac{\mu_n - r}{r}\right)\right)^{-K/2} \\
        &\lesssim (\lambda_n r)^{-K} \left(\frac{\mu_n-r}{r}\right)^{-K} \left(\mu_n - r\right)^{-K/2} \\
        &\lesssim \lambda_n^{-K} (\mu_n -r)^{-3K/2} \\
        &\lesssim \lambda^{-3/2\kappa K},
    \end{split}\]
    from which the claim follows by choosing $K$ large enough depending on $N,\kappa$.
\end{proof}

\subsection{Derivatives of $J_n(\lambda \cdot)$ near the caustic}\label{Appcaus}

\begin{lemma}\label{derlemm}
    There exists a constant $C > 0$ such that, for all $n \in Z$ such that $\mu_n \gtrsim \lambda^{-2/3 + \alpha}$, and for all $r\in [0,1]$ such that $|r- \mu_n| \lesssim \lambda^{-2/3+ \kappa}$, there holds
    \[\left|\frac{d^K}{dr^K} \left(\lambda_n^{1/2} J_n(\lambda_n r)\right)\right| \lesssim_{K,\alpha,\kappa} \lambda^{C + \frac{2}{3} K + \frac{1}{2}\kappa K} r^{-\frac{K}{2}}.\] 
\end{lemma}

\begin{proof}
    Let us start with the expression
    \[\frac{d^K}{dr^K} \left(\lambda^{1/2}_{k,n} J_n(\lambda_{k,n} r)\right) = \lambda_n^{1/2 + K} \int_{\mathbb{T}} e^{i\lambda_n (\mu_n \alpha - r \sin \alpha)} (-\sin \alpha)^K d\alpha.\]
    \myindent Now, define the phase function
    \[\phi : \alpha \mapsto \frac{\mu_n}{r} \alpha - \sin \alpha.\]
    \myindent Observe that $\frac{\mu_n}{r} = 1 + \frac{\mu_n - r}{r}$, where $\left|\frac{\mu_n -r}{r}\right|\lesssim \eps := \lambda^{-\frac{2}{3} + \kappa} r^{-1} \lesssim \lambda^{\kappa - \alpha} << 1$. Thus, writing
    \[|\phi'(\alpha)| = \left|\frac{\mu_n}{r} - \cos \alpha\right| = \left| 1 - \cos \alpha + \frac{\mu_n - r}{r}\right|,\]
    we find that, if $C > 0$ is a large enough constant, for $|\alpha| \geq C\sqrt{\eps}$, there holds 
    \[|\phi'(\alpha)| \gtrsim \eps, \qquad \left|\frac{\phi''(\alpha}{\phi'(\alpha)}\right| \lesssim \eps^{-1/2}.\]
    \myindent Thus, if we define $\zeta(\alpha)$ a smooth localizer on $|\alpha|\leq C \sqrt{\eps}$, such that $|\zeta^{(k)}(\alpha)|\lesssim \eps^{-k/2}$ for all $k$, we may integrate by parts $M$ times on the support of $1 - \zeta$ to find that
    \[\int_{\mathbb{T}} e^{i\lambda_n r \phi(\alpha)} (-\sin \alpha)^K (1 - \zeta(\alpha) d\alpha = \int_\mathbb{T} O((\lambda_n r \eps)^{-M} ) \eps^{-M/2} d\alpha.\]
    \myindent Now, write
    \[\lambda_n r \eps^{3/2 M} = \lambda_n r \lambda_n^{-1 + \frac{3}{2}\kappa} r^{-3/2} \gtrsim \lambda^{3/2 \kappa}.\]
    
    \myindent In particular, as $M\to \infty$, we find that
    \[\frac{d^K}{dr^K} \left(\lambda^{1/2}_{k,n} J_n(\lambda_{k,n} r)\right) = \lambda_n^{1/2 + K} \int_\mathbb{T} e^{i\lambda_n r \phi(\alpha)} (-\sin \alpha)^K \zeta(\alpha) d\alpha + O(\lambda^{-\infty}).\]
    \myindent To conclude, one needs only observe that, on the support of $\zeta$, there holds $|\sin \alpha|\leq |\alpha|\lesssim \sqrt{\eps}$. Overall, we find that
    \[\left|\frac{d^K}{dr^K} \left(\lambda^{1/2}_{k,n} J_n(\lambda_{k,n} r)\right)\right| \lesssim \lambda_n^{1/2 + K} \eps^{K/2}.\]
    \myindent Now, observe that
    \[\lambda_n^{1/2 + K} \eps^{K/2} = \lambda_n^{1/2 + K} \lambda^{-1/3 K + \frac{K}{2}\kappa} r^{-K/2},\]
    which concludes the proof.
\end{proof}

\subsection{Oscillatory behavior in the exterior region : proof of Lemma \ref{Lemmstatphase}}\label{applemm1}

\myindent Recall first that $\lambda_n^{1/2} J_n(\lambda_n r)$ is given by the renormalized oscillatory integral
\[\lambda_n^{1/2}J_n(\lambda_n r) = \lambda_n^{1/2}\int_{\mathbb{T}} \exp(i\lambda_n(\mu_n \alpha - r\sin \alpha)) d\alpha,\]
which may be written, with the change of large parameter $\lambda_n \to \lambda_n r \gtrsim \lambda^{1/3 +\alpha}$,

\[\lambda_n^{1/2} \int_{\mathbb{T}} \exp\left(i\lambda_n r \left(\frac{\mu_n}{r}\alpha - \sin \alpha \right)\right) d\alpha.\]

\myindent Now, the phase
\[\phi : \alpha \mapsto \frac{\mu_n}{r} \alpha - \sin \alpha\]
has two stationary points 
\[\pm \alpha_n(r) = \pm \arccos\left(\frac{\mu_n}{r}\right),\]
at which there holds
\[\begin{split}
    |\phi''(\pm\alpha_n)| = \sin(\alpha_n) = \sqrt{1 - (\mu_n / r)^2},
\end{split}\]
hence, the stationary phase lemma yields that, to the main order, there should hold
\[\lambda_n^{1/2}J_n(\lambda_n r) \simeq \frac{\cos\left(\lambda_n r \left(\frac{\mu_n}{r}  \alpha_n - \sin \alpha_n\right)\right))}{r^{1/2} \left(1 - (\mu_n/r)^2\right)^{1/4}} = \frac{\cos(\lambda_n(\mu_n \alpha_n - r\sin(\alpha_n)))}{(r^2 - \mu_n^2)^{1/4}}.\]

\myindent Thus, the difficulty comes from the need to carefully estimate the remainder term and its derivatives in $r$. Without loss of generality, we first carry out the analysis near the stationary point $\alpha_n$. We start by finding an appropriate change of variable near $\alpha_n$. For $h$ small, the phase function writes as
\[\phi(\alpha_n + h) - \phi(\alpha_n) = \frac{1}{2}\gamma h^2 + O(h^3),\]
where $\gamma(r, n) := \phi''(\alpha_n(r)) = \sqrt{1 - (\mu_n/r)^2}$ can be very small. On the contrary, the third derivative $\phi^{(3)}(\alpha_n)$ is typically of order 1. Hence, instead of the usual change of variable $\phi(\alpha_n + h) -\phi(\alpha_n) = \frac{1}{2}\gamma y^2$, we first introduce the change of variable $h= \gamma u$ so that there holds
\[\phi(\alpha_n + h) - \phi(\alpha_n) = \frac{1}{2}\gamma^3(u^2+ O(u^3)).\]

\myindent Now, it is natural to introduce the change of variable $y(u)$ so that
\[\phi(\alpha_n + h) - \phi(\alpha_n) = \frac{1}{2}\gamma^3(u^2+ O(u^3)) =\frac{1}{2} \gamma^3 y(u)^2.\]

\myindent Observe that $y(u)$ is given through Taylor formula by
\[y = u \left(2\int_0^1 (1 - t) \gamma^{-1} \phi''(\alpha_n + t \gamma u) dt \right)^{1/2}.\]

\myindent The interest is that this immediately yields that $u \mapsto y(u)$ is a smooth change of variable for $|u|\lesssim 1$, such that $|y^{(k)}(u)| \lesssim_k 1$ independently of $\gamma,\alpha_n$, and $y'(0) = 1$. In particular, one may invert the change of variable close to $0$, and find $\chi(y)$ a smooth localizer close enough to $0$ (independent of all parameters), which equals $1$ close to $0$, such that $y \mapsto u(y)$ is a smooth change of variable on $supp(\chi)$, with all its derivatives in $y$ bounded independently of $\gamma, \alpha_n$. Hence, letting
\[\zeta(\alpha) := \chi\left(y\left(\frac{\alpha - \alpha_n}{\gamma}\right)\right),\]
which is a smooth localizer supported close to $\alpha_n$, and which equals $1$ on an interval of size $\simeq \gamma$ around $\alpha_n$, we may write
\[\lambda_n^{1/2}\int e^{i\lambda_n r (\phi(\alpha)-\phi(\alpha_n))} \zeta(\alpha) = \lambda_n ^{1/2} \int e^{i\frac{1}{2}\lambda_n r \gamma^3 y^2} \chi(y) \gamma u'(y) dy.\]

\myindent Now, thanks for example to \cite{hormanderanalysis}[Theorem 7.7.5], there holds, for any $K \geq 0$,

\[\lambda_n^{1/2} \int e^{i\frac{1}{2}\lambda_n  r \gamma^3 y^2} \chi(y) \gamma u'(y) dy = \gamma^{-1/2} r^{-1/2} \left(\sum_{k=0}^K (\lambda_n r \gamma^3)^{-k} A_{2k}[u'](0) + O((\lambda_n r \gamma^3)^{-(K+1)}) \right),\]
where $A_{2k}$ are smooth differential operators of order $2k$ with universal constant coefficients, and $A_0[u'](0) = u'(0) = 1$. Now, observe that
\[\lambda_n r \gamma^3 = \lambda_n r (1-(\mu_n/r)^2)^{3/2} = \lambda_n r^{-2} (\mu_n + r)^{3/2} (\mu_n - r)^{3/2} \geq \lambda_n r^{-1/2}(\mu_n -r)^{3/2} \gtrsim r^{-1/2} \lambda^{\frac{3}{2}\kappa} \gtrsim \lambda^{\frac{3}{2}\kappa},\]
hence, for any fixed $N,\kappa$, there is a $K$ large enough so that the remainder is $O(\lambda^{-N})$. In particular, if, with this choice of $K$, we define 
\[a_{n,N,\lambda} (r) := 2\sum_{k=1}^K (\lambda_n r \gamma^3)^{-k} A_{2k}[u'](0),\]
we find that
\[\lambda_n^{1/2}\int e^{i\lambda_n r \phi(\alpha)} \zeta(\alpha) = \frac{e^{i\lambda_nr \phi(\alpha_n)}}{2} a_{n,N,\lambda}(r) + O_{\kappa,N}(\lambda^{-N}).\]

\myindent Similarly, for $\alpha < 0$, since $\alpha \mapsto \zeta(-\alpha)$ localizes around $-\alpha_n$, the same analysis yields that
\[\lambda_n^{1/2}\int e^{i\lambda_n r \phi(\alpha)} \zeta(-\alpha) d\alpha = \frac{e^{-i\lambda_nr \phi(\alpha_n)}}{2 \gamma^{1/2} r^{1/2}} a_{n,N,\lambda}(r) + O_{\kappa,N}(\lambda^{-N}),\]
and hence
\[\lambda_n^{1/2}\int e^{i\lambda_n r \phi(\alpha)} (\zeta(\alpha) + \zeta(-\alpha))d\alpha = \frac{\cos(\lambda_n r \phi(\alpha_n))}{r^{1/2}\gamma^{1/2}} a_{n,N,\lambda}(r),\]
which almost yields the result. The only point to check is that
\[\int e^{i\lambda_n r \phi(\alpha)} (1 - (\zeta(\alpha) + \zeta(-\alpha))) d\alpha = O_{\kappa}(\lambda^{-\infty}).\]

\myindent However, this is straightforward from integration by parts in $\alpha$. Indeed, observe that, on the support of $(1 - (\zeta(\alpha) + \zeta(-\alpha)))$, there holds uniformly $|\phi'(\alpha)|\gtrsim \gamma^2$, and $\frac{|\phi''(\alpha)|}{|\phi'(\alpha)|} \lesssim \gamma^{-1}$. Moreover, there also holds $|\zeta^{(k)}(\alpha)|\lesssim \gamma^{-k}$. Hence, when integrating by parts $M$ times in $\alpha$, we find that
\[\left| \int e^{i\lambda_n r \phi(\alpha)} (1 - (\zeta(\alpha) + \zeta(-\alpha))) d\alpha \right| \lesssim \int \frac{1}{\lambda^{M} r^{M} \gamma^{2M}} \gamma^{-M} \lesssim (\lambda r \gamma^3)^{-M},\]
which, with the same analysis than above, is a $O(\lambda^{-\infty})$.

\quad

\myindent To conclude the proof of the lemma, we need only prove estimate \eqref{deranearcaus} on $a_{n,N,\lambda}(r)$. First, for $K = 0$, observe that, since there holds uniformly $\lambda r \gamma^3 \gtrsim 1$, and since the derivatives of $u$ in $y$ are uniformly bounded, it is straightforward that $|a_{n,N,\lambda}(r)|\lesssim_\kappa 1$.

\myindent Now, when differentiating on $r$, the only unbounded quantities which may appear are derivatives of $r^{-k}$ and negative powers of derivatives in $r$ of $\gamma$. However, observe that the dependency in $r$ of all such quantities is given by sums and product of powers of $r$, of $r+\mu_n$, and of $r-\mu_n$. Using moreover Leibniz rule, and, since $r-\mu_n$ is the smallest of all three above quantities, we find that, when differentiating $K$ times, the worst quantity which appears is $(r-\mu_n)^{-K}$, hence the upper bound.

\section{Extension to surfaces isometric under a $S^1$ action}\label{Appext}

\myindent In this section, we give the outline of the proof of Theorem \ref{thmext}, that is we explain how to reduce the result to a number of estimates which can then be proved using the same methods than for the Euclidean disk.

\myindent Following for example \cite{donnelly1978g}, let us split $L^2(\Sigma)$ into the direct sum $L^2(\Sigma) = \bigoplus_{n\in\Z} L^2_n(\Sigma)$, where $L^2_n(M)$ transforms according to the character $e^{in\theta}$ of $S^1$ under the group action of $S^1$. We may thus find a joint basis of eigenfunction of $(\sqrt{-\Delta},P)$ of the form $\phi_{k,n}$, $k\in \N$, $n\in \Z$, with joint eigenvalues $(\lambda_{k,n},n)$, where, for any $n$, $(\phi_{k,n})_{k\geq 1}$ is an orthonormal basis of $L^2_n(\Sigma)$ which diagonalizes the restriction of $\Delta$ to $L^2_n(\Sigma)$. From Sturm-Liouville theory, the joint eigenvalues are simple. Now, since the fixed points of the action of $S^1$ are isolated, they form a finite set, say $F$. Thus, there are polar coordinates of the form $(\theta,r)$ on the open set $\Sigma\backslash F$ (one may choose $r$ a coordinate on the orbit space $\Sigma /S^1$ where singular orbits are removed), for which there holds the separation of variables
w\[\phi_{k,n}(\theta,r) = e^{in\theta} \Phi_{k,n}(r),\]
and for which the Riemannian volume form writes locally $v(r) dr d\theta$. 

\myindent Let us thus fix $\omega$ and $\chi(x,D)$ as in Theorem \ref{thmext}. Observe that, since $\chi$ is supported away from lagrangian tori with maximally degenerate caustics, there holds $\pi(supp(\chi)) \subset \Sigma \backslash F$, where $\pi : T^*\Sigma \to \Sigma$ is the canonical projection. In particular, we may reduce the problem to the following : if $\zeta(r)$ is a localizer on an arbitrarily small compact subset of $\Sigma \backslash F$, which is the reunion of $S^1$ orbits, then we need only prove that $\|\zeta P_{\lambda,\delta} \chi(x,D)\|_{L^2(\Sigma) \to L^4(\Sigma)} \lesssim \lambda^{1/8} \delta^{1/8}\ln(\lambda)^{1/4}$.

\myindent Observe that, by decomposing further $\chi(x,D) = \sum_{i=1}^K \chi_i(x,D)$, for some partition of unity, and using the triangular inequality, we may without loss of generality assume that $\chi(x,D)$ is supported on an arbitrarily small conic open subset of $T^*\Sigma$ (as long as it is independent of $\lambda,\delta$). In particular, we may assume that the cone $C_1$ in Definition \ref{defadapted} is a small conic neighborhood of a direction, say $\R_+^*\xi_0$. Thus, if $M=(p,\Theta)$ is the moment application on $T^*\Sigma$, we see that $\chi$ is supported on $M^{-1}\{(K(\xi), \xi_2) \qquad \xi \in C_1\}$. We denote $\tilde{C} := \{(K(\xi), \xi_2) \qquad \xi \in C_1\}$.

\myindent Now, from standard semiclassical analysis, for any $(k,n)\in \Z \times \N$, if $(\lambda_{k,n},n) \notin \tilde{C}$, then $\|\chi(x,D) \phi_{k,n}\|_{L^{\infty}(\Sigma)} = O(\lambda_{k,n}^{-\infty})$. Hence, if we denote, in the spirit of Section \ref{secnotations},
\[\mathcal{Z} := \{(k,n) \qquad \text{such that} \ \lambda_{k,n}\in[\lambda- \delta,\lambda + \delta] \ \text{and} \ (\lambda_{k,n},n) \in \tilde{C}\},\]
we need only prove that, for any $u := \sum_{(k,n)\in \mathcal{Z}} a_{k,n}\phi_{k,n}$ such that $\sum |a_{k,n}|^2 = 1$, there holds
\[\|\zeta u \|_{L^4(\Sigma)}^4 \lesssim \lambda^{1/2} \delta^{1/2}\ln(\lambda).\]

\myindent Now, observe that, since $\sqrt{-\Delta}$ has vanishing subprincipal symbol, the Born-Oppenheimer quantification conditions yields that there is a Maslov index $\mu \in \frac{1}{4}\Z$ such that, for any $(k,n) \in \mathcal{Z}$, one can write
\[\lambda_{k,n} = K(m_n + \mu, n) + O(\lambda_{k,n}^{-1}),\]
where $m_n\in \N$ is the unique integer such that $K(m_n + \mu, n) \in [\lambda -2\delta, \lambda + 2\delta]$ (for example). Thus, we may parameterize $\mathcal{Z}$ by its projection on the second variable $Z$, and write the elements of $\mathcal{Z}$ as $(k_n, n)$, for $n\in Z$. Observe that, from the curvature assumption on the curve $\{K = 1\}$, and from \cite{colin1977nombre}, the cardinal of $Z$ is of order $O(\lambda \delta)$, as long as $\delta \geq \lambda^{-1/3}$.

\myindent Now, following the same computations than in Section \ref{secnotations}, there holds, dropping the $k_n$ index in the $a_{k,n}$

\[\|\zeta u\|_{L^4(\Sigma)}^{4} = \sum_{n_1,n_2,n_3,n_4 \in Z} a_{n_1}\overline{a_{n_2} a_{n_3}} a_{n_4} \int v(r) dr \Pi_{i=1}^4 \Phi_{k_{n_i},n_i}(r) \int_{\mathbb{T}} e^{i(n_1 - n_2 - n_3 + n_4) \theta} d\theta.\]

\myindent In particular, the theta integral vanishes if the $(n_i)$ don't satisfy the zero sum set condition, as in Section \ref{subsectrapezes}. We may thus introduce the same set $T$ of zero sum sets, and observe that we need only prove that, if
\[C_{\textbf{n}} := \int v(r) dr \Pi_{i=1}^4 \Phi_{k_{n_i},n_i}(r),\]
then there holds
\[\sum_{\textbf{n} \in T} |a_{n_1}| |a_{n_2}||a_{n_3}||a_{n_4}| |C_{\textbf{n}}| \lesssim \lambda^{1/2} \delta^{1/2}\ln(\lambda).\]

\quad

\myindent The proof is now extremely similar, since the main ingredients still holds in that model

i. First, the cardinal of $Z$ is bounded by $O(\lambda \delta)$.

\quad

ii. Second, the quadrilateral Jarnik's Lemma in Section \ref{subsecuncertainty} still holds, using the one-to-one correspondence between points of $Z$, and integer points $(m,n)$ which are in a small neighborhood of size $O(\delta)$ of the curve $\{K = \lambda\}$ which has a nonvanishing curvature.

\quad 

iii. Third, for $n\in Z$, the eigenfunction $\phi_{k_n,n}$ is an oscillatory function of order zero associated to the Lgrangian torus $\ell_n := M^{-1}(1, n/\lambda_{k_n,n})$, with semiclassical parameter $\lambda_{k,n}^{-1}$ see for example \cite{san2006systemes}[Proposition 3.2.12]. In particular, modulo $O(\lambda^{-\infty})$ remainders, $\Phi_{k_n,n}(r)$ is described by standard semiclassical BKW analysis for a nondegenerate potential well, with semiclassical parameter $\lambda_{k_n,n}^{-1}$.

\quad

iv. Last, one has the same type of estimates on the coefficients $C_{\textbf{n}}$. Indeed, there are essentially two cases.

a) Either, for any $\xi \in C_1$, the caustic set of $L_\xi$ is at a positive distance of $supp(\chi)$. Then, $\phi_{k_n,n}$ writes on the support of $\zeta$ as a standard BKW function, since it is an oscillatory function on a Lagrangian submanifold evaluated away from caustics. One may then use nonstationary phase arguments to prove for any $\alpha >0$, and any zero sum set $\textbf{n} \in T$, if $s(\textbf{n}), t(\textbf{n}) \gtrsim \lambda^\alpha$, then $C_{\textbf{n}} = O(\lambda^{-\infty})$. Moreover, the $L^4$ norm of $\phi_{k_n,n}$ is uniformly bounded, thus $C_\textbf{n} = O(1)$ for any $\textbf{n}$. Hence, the counting estimates are in fact even better. We won't give a more detailed proof for these estimates since they don't yield the main order term, and since they are quite universal.

b) Either we may use the universality of the Airy function to describe the transition between classically allowed and classically forbidden region in order to express $\Phi_{k_n,n}$. Precisely, this will yield that, modulo a small remainder, one can write
\[\Phi_{k_n,n}(r) \simeq Ai(\rho(r,n) \lambda_{k_n,n}^{2/3}) a(r,n),\]
where $a(r,n), \rho(r,n)$ are smooth symbols uniformly bounded in $r,n$ along with all their derivatives, and $\rho(r,n)$ satisfies $\rho = 0$ only at the caustic, and $\partial_r \rho \neq 0$. In particular, the point is that $\rho(r,n) \simeq \pm (r- r_n)$, if we denote by $\{r= r_n\}$ the caustic set of $\ell_n$. The point is that $\Phi_{k_n,n}(r)$ thus satisfies exactly the same estimates than the functions $\lambda_n^{1/2}J_n(\lambda_n r)$ on the disk, i.e. its $L^4$ norm is $O(\ln(\lambda))$ from a direct computation, and there is a tripartition such that, up to changing a sign

- if $r-r_n \lesssim - \lambda^{-2/3 + \kappa}$, then $\Phi_{k_n,n}(r)$ is a $O(\lambda^{-\infty})$.

- if $|r-r_n| \lesssim \lambda^{-2/3 + \kappa}$, then the derivatives in $r$ of $\Phi_{k_n,n}$ grow like powers of $\lambda^{2/3 + \kappa}$.

- if $r - r_n \gtrsim \lambda^{-2/3 + \kappa}$, one may use the asymptotic development of the Airy function to obtain the equivalent of Lemma \ref{Lemmstatphase}, that is $\Phi_{k_n,n}(r)$ is an oscillatory function of the form
\[\cos(\lambda_{k_n,n} f(r,n)) a_n(r),\]
where $|(\partial_r^k)a_n(r)| \lesssim (r- r_n)^{-k}$, and where \[f(r,n) = c \rho(r,n)^{3/2}\]
for some constant $c$. The point is that this function satisfies exactly the same convexity estimates and derivatives bounds than the function $f(r,\nu)$ which appeared above. Indeed, when differentiating more than two times, the dominant term will be of the form $(\partial_r \rho)^K \rho(r,n)^{3/2 - K}$, which is larger than all other terms provided the support of $\zeta$ is small enough, since it scales like an inverse power of the distance $(r- r_n)$. 

\printbibliography

\end{document}